\documentclass[a4paper,11pt]{amsart}
\usepackage[utf8]{inputenc}
\usepackage[T1]{fontenc}
\usepackage[UKenglish]{babel}
\usepackage[margin=18mm]{geometry}

\usepackage{color}

\usepackage{graphicx}

\usepackage{amsmath,amssymb,amsfonts,amsthm}
\usepackage{mathrsfs,eucal,dsfont}
\usepackage{verbatim,enumitem}

\usepackage{hyperref,url}

\renewcommand{\geq}{\geqslant}

\newcommand{\R}{\mathds R}

\newcommand{\sL}{\mathcal{L}}

\def\<{\langle}
\def\>{\rangle}

\def\R{\mathbb R}

\def\<{\langle} \def\>{\rangle}

 \def\beq{\begin{equation}}

\def\to{\rightarrow}
\def\8{\infty}\def\3{\triangle}
\def\1{\lesssim}

\renewcommand{\tilde}{\widetilde}

\newtheorem{theorem}{Theorem}[section]
\newtheorem{lemma}[theorem]{Lemma}
\newtheorem{proposition}[theorem]{Proposition}

\theoremstyle{definition}
\newtheorem{definition}[theorem]{Definition}

\newtheorem{remark}[theorem]{Remark}

\numberwithin{equation}{section}
\begin{document}
\allowdisplaybreaks
\title[
Local boundedness of parabolic fractional $p$-Laplacian type equations] {
Local boundedness of solutions to parabolic equations associated with fractional $p$-Laplacian type operators}

\author{Takashi Kumagai\qquad Jian Wang  \qquad \hbox{ and }\qquad
Meng-ge Zhang}

 \date{}

\maketitle

\begin{abstract}
In this paper, we study the local boundedness of local weak solutions to the following parabolic equation associated with fractional $p$-Laplacian type operators
$$
\partial_t u(t,x)-\text{p.v.}\int_{\R^d}|u(t,y)-u(t,x)|^{p-2}(u(t,y)-u(t,x))J(t;x,y)\,dy=0,\quad (t,x)\in \R\times \R^d,
$$
where $\text{p.v.}$ means the integral in the principal value sense, $p\in(1,\infty)$ and $J(t;x,y)$ is comparable to the kernel of the fractional $p$-Laplacian operator $|x-y|^{-d-sp}$ with $s\in(0,1)$ and uniformly in $(t;x,y)\in\R\times\R^d\times\R^d$.
Unlike existing results in the literature, the local boundedness of the solutions obtained in this paper extends the known results
for the linear case (i.e., the case that $p=2$), in particular with a nonlocal parabolic tail that uses the $L^1$-norm in time for all $p\in (1,\infty)$. The proof is based on a new
level set truncation in the De Giorgi-Nash-Moser iteration
 and a careful choice of iteration orders,
  as well as a general Caccioppoli-type inequality that is efficiently
applied to fractional $p$-Laplacian type operators with all $p>1$.

\medskip


\medskip

\end{abstract}
\allowdisplaybreaks

\section{Introduction and main result}
The purpose of this paper is to establish the local boundedness of
subsolutions to the following equation
\begin{equation}\label{EQ}
\frac{\partial u(t,x)}{\partial t}-\sL_tu(t,x)=0,\quad (t,x)\in I\times D\subset \R\times \R^d,
\end{equation}
where $I$ is a bounded interval of $\R$, $D$ is a bounded domain of $\R^d$, and $\sL_t$ is a nonlocal (nonlinear) operator of the fractional $p$-Laplacian type with $p>1$; that is,
$$\sL_tu(x)=\text{p.v.}\int_{\R^d}|u(t,y)-u(t,x)|^{p-2}(u(t,y)-u(t,x))J(t;x,y)\,dy,\quad (t,x)\in\R\times\R^d.$$
Here, the symbol $\text{p.v.}$ means the integral in the sense of the principal value, and the kernel $J(t;x,y)$ satisfies that for any $t>0$ and $x,y\in \R^d$, $J(t;x,y)=
J(t;y,x)$, and that there are $\Lambda\ge1$ and $s\in (0,1)$ so that for all $t\in \R$ and $x,y\in \R^d$,
\begin{equation}\label{J0}
 \frac{\Lambda^{-1}}{|x-y|^{d+sp}}\le J(t;x,y)\le\frac{\Lambda}{|x-y|^{d+sp}}. \end{equation} When $p=2$ and
$J(t;x,y)=|x-y|^{-d-2s}$, the nonlocal operator $\sL_t$ boils down to the well-known fractional Laplacian $-(-\Delta)^s$  multiplied by a proper constant,
which is a linear operator.
Note that the operator $\sL_t$ corresponds to the following nonlocal (nonlinear) $p$-form
$$\mathcal{E}_t(f,g)=\frac{1}{2}\iint_{\R^d\times\R^d} |f(y)-f(x)|^{p-2}(f(y)-f(x))(g(y)-g(x))J(t;x,y)\,dx\,dy.$$

\subsection{Notations} For any measurable function $f(t,x)$ on $I\times D$,
$$\oint_{I}\oint_{D}
 f(t,x) \,dx\,dt:= \frac{1}{|I\times D|}\int_{I}\int_{D}
 f(t,x) \,dx\,dt.$$
Fix $p>1$.
For a domain $D\subset\R^d$, the fractional Sobolev space $W^{s,p}(D)$ consists of all functions $f\in L^p(D;dx)
=: L^p(D)$ such that
$$[f]_{W^{s,p}(D)}:=\left(\int_D\int_D\frac{|f(x)-f(y)|^p}{|x-y|^{d+sp}}\,dx\,dy\right)^{{1}/{p}}<\infty.$$
The norm of $f\in W^{s,p}(D)$ is given by
$$\|f\|_{W^{s,p}(D)}=[f]_{W^{s,p}(D)}+\|f\|_{L^p(D)},$$ where $\|f\|_{L^p(D)}=\left(\int_D |f(x)|^p\,d x\right)^{1/p}.$

For $q\in (0,\infty]$, the parabolic Sobolev space $L^q(0,T;W^{s,p}(D))$ is the set of measurable functions on $(0,T)\times D$ such that
$$\|f\|_{L^q(0,T;W^{s,p}(D))}=\left(\int_0^T\|f(t,\cdot)\|^q _{W^{s,p}(D)}\,dt\right)^{{1}/{(q\vee1)}}<\infty.$$
In particular, when $q=\infty$,  $\|f\|_{L^\infty(0,T;W^{s,p}(D))}=\sup_{t\in [0,T]} \|f(t,\cdot)\| _{W^{s,p}(D)}$.
\subsubsection{Subsolution}
We call that $u(t,x)$ is a {\it subsolution} (resp.\ {\it supsolution}) associated with the operator $\sL_t$ on $I\times D\subset \R\times \R^d$, if  the following equation holds
\begin{equation}\label{PE} \frac{\partial u(t,x)}{\partial t}-\sL_t u(t,\cdot)(x)\le0\,\, (\hbox{resp.} \ge0),\quad (t,x)\in I\times D. \end{equation}
Furthermore, $u(t,x)$ is a \textit{solution} associated with  the operator $\sL_t$ on $I\times D\subset \R\times \R^d$, if$$ \frac{\partial u(t,x)}{\partial t}-\sL_t u(t,\cdot)(x)=0,\quad (t,x)\in I\times D. $$
The existence and the
uniqueness of the solution associated with  the operator $\sL_t$ on $I\times D\subset \R\times \R^d$ is proved in \cite[Section 2.1.1]{St1}. Some properties and estimates of subsolutions to the operator $\sL_t$ on $I\times D$
are also considered in \cite[Section 3]{St1}. In particular, we have the following simple lemma, see \cite[Lemma 3.1]{St1}.
\begin{lemma}\label{Z}
If $u$ is a subsolution  associated with  the operator $\sL_t$ on $I\times D\subset \R\times \R^d$, then $u_+:=\max\{u,0\}$ is also a subsolution.
\end{lemma}

\subsubsection{Parabolic tail}

To state the local boundedness for subsolutions to the equation \eqref{EQ}, we also need the following nonlocal quantity, which is called the parabolic tail of the function.

\begin{definition}\it Let $p>1$ and $q\in (0,\infty]$.
For any $x_0\in\R^d$, $t_0\in \R$, $r>0$, $T_0>0$ and any function $u\in L^q (t_0-T_0,t_0;W^{s,p}(\R^d))$, the parabolic tail of $u$ with respect to $x_0, r, t_0$ and $T_0$ is defined by
\begin{equation}\label{Tail1}
{\rm Tail}_q
(u; x_0, r, t_0-T_0,t_0)=\left(\frac{1}{T_0}\int_{t_0-T_0}^{t_0}\left(r^{sp}\int_{\R^d\setminus B_{r}(x_0)}\frac{
u_+(t,x)^{p-1
}}{|x-x_0|^{d+sp}}\,dx\right)^{q/(p-1)}\,dt\right)^{1/q},
\end{equation} where $ B_r(x_0)=\{x\in \R^d:|x-x_0|<r\}$.
In particular,
$$
{\rm Tail}_{p-1}(u; x_0, r, t_0-T_0,t_0)=\left(\frac{r^{sp}}{T_0}\int_{t_0-T_0}^{t_0} \int_{\R^d\setminus B_{r}(x_0)}\frac{
u_+(t,x)^{p-1
}}{|x-x_0|^{d+sp}}\,dx \,dt\right)^{1/(p-1)};$$
and the parabolic supremum tail of $u$ with respect to $x_0, r, t_0$ and $T_0$ is given by
\begin{equation}\label{Tail2}
{\rm Tail}_\infty(u; x_0, r, t_0-T_0,t_0)=\left(r^{sp}\sup_{t_0-T_0< t< t_0}\int_{\R^d\setminus B_{r}(x_0)}\frac{
u_+(t,x)^{p-1
}}{|x-x_0|^{d+sp}}\,dx\right)^{1/(p-1)}.
\end{equation}
\end{definition}

\subsection{Main result}
We now state the main result of this paper, which shows the local boundedness of
subsolutions to the parabolic equation associated with fractional $p$-Laplacian type operators.
In the following, set
$$I^\ominus_r(t_0)=(t_0-r^{sp},t_0)
\quad \text{and}\quad  B_r(x_0)=\{x\in \R^d:|x-x_0|<r\}.$$

\begin{theorem}\label{main1} Let $p\in (1,\infty)$, and $u$ be a
subsolution to $\partial_tu-\sL_tu=0$ in $I^\ominus_R(t_0)\times B_R(x_0)$ with $t_0\in \R$, $x_0\in\R^d$ and $R>0$,
Then, the following statement hold.
\begin{itemize}
\item[{\rm (i)}] For any $p\in [2,\infty)$, there is a constant $C_1>0$ so that
\begin{align*}
&\sup_{B_{R/2}(x_0)\times I^\ominus_{R/2}(t_0)}u\\
&\le C_1\left(\oint_{I^\ominus_{R}(t_0)}\oint_{B_{R}(x_0)}u_+(t,x)^{2p-2}\,dx\,dt+\oint_{I^\ominus_{R}(t_0)}\oint_{B_{R}(x_0)}u_+(t,x)^p\,dx\,dt\right)^{{(p-1)}/{p}}\\
&\quad\,\,\,+C_1\left[1+ \left(\oint_{I^\ominus_R(t_0)}\oint_{B_R(x_0)}u_+(t,x)^{p-1}\,dx\,dt\right)^{2-p}\right]^
{(1+d/(sp))/2}\left(\oint_{I^\ominus_R(t_0)}\oint_{ B_R(x_0)}u_+(t,x)^p\,dx\,dt\right)^{1/2}\\
&\quad\,\,\,+
C_1{\rm Tail}_{p-1}^{p-1}(u; x_0,R, t_0-{R}^{sp}, t_0)+
C_1\left({\rm Tail}_{p-1}^{p-1}(u; x_0,R, t_0-R^{sp}, t_0)\right)^{p-1}.
\end{align*}

\item[{\rm (ii)}] For any $p\in (1,2)$ and $\xi>d(2-p)/(sp)$, there is a constant $C_2>0$ so that
\begin{align*}
&\sup_{B_{R/2}(x_0)\times I^\ominus_{R/2}(t_0)}u\\
&\le C_2\left(\oint_{I^\ominus_{R}(t_0)}\oint_{B_{R}(x_0)}u_+(t,x)^{2p-2}\,dx\,dt+\oint_{I^\ominus_{R}(t_0)}\oint_{B_{R}(x_0)}u_+(t,x)^p\,dx\,dt\right)^{1/p}\\
&\quad\,\,\,+C_2\left[1+ \left(\oint_{I^\ominus_R(t_0)}\oint_{B_R(x_0)}u_+(t,x)^{p-1}\,dx\,d t\right)^{(p-2)/(p-1)}\right]^{(1+\beta)/((1+\xi)\theta)}\\
&
\qquad\qquad\times \left(\oint_{I^\ominus_R(t_0)}\oint_{ B_R(x_0)}u_+(t,x)^{1+\xi}\,dx\,dt\right)^{\beta/((1+\xi)\theta)}\\
&\quad\,\,\,+C_2{\rm Tail}_{p-1}^{p-1}(u; x_0,R, t_0-{R}^{sp}, t_0),
\end{align*}
where $\theta=\frac{(p-1+\xi)(1+sp/d)}{1+\xi}-1$ and  $\beta=sp(p-1+\xi)/(d(1+\xi))$.
\end{itemize}
\end{theorem}

\begin{remark}
 \begin{itemize}
\item[{\rm (i)}] For the equation $\partial_tu-\sL_tu=f$ involving a source term $f\in L^{1,\infty}_{t,x}(I^\ominus_{R}(t_0)\times B_{R}(x_0))$, we observe that if $u$ is a subsolution to $\partial_tu-\sL_tu=f$, then the function
$$v(t,x):=u(t,x)-\int^t_{t_0}\|f(s,\cdot)\|_{L^\infty(B_{R}(x_0))}\,ds$$
is a subsolution to $\partial_tv-\sL_tv=0$. Therefore, one can get the local boundedness of
subsolutions $u$ to the equation $\partial_tu-\sL_tu=f$ by directly applying Theorem \ref{main1}.

\item[{\rm(ii)}] When $p=2$, the assertion in Theorem
    \ref{main1}\,(i) is reduced
     to
$$\sup_{B_{R/2}(x_0)\times I^\ominus_{R/2}(t_0)}u\le C_0\left(\left(\oint_{I^\ominus_{R}(t_0)}\oint_{B_{R}(x_0)}u_+(t,x)^2\,dx\,dt\right)^{1/2}+{\rm Tail}_{1}(u; x_0,R, t_0-{R}^{sp}, t_0)\right),$$ which is consistent with \cite[Theorem 1.8]{KW2} in the linear setting. On the other hand, when $p=2$, one can take $\xi=1$ in Theorem
\ref{main1}\,(ii), and so $\theta=\beta=2s/d$. In particular, in this case Theorem
\ref{main1}\,(ii) is also simplified into the statement above.

\item[{\rm(iii)}]
To the best of our knowledge, quantitative estimates for subsolution to the parabolic equation associated with fractional $p$-Laplacian type operators stated in Theorem \ref{main1} have not appeared in the literature before, and Theorem \ref{main1} is new in the sense that it is reduced into the linear case  when $p=2$
(see Remark (ii) above and \cite[Theorem 1.8]{KW2}).
 Indeed, there is an additional (constant) term in all the known results for general
fractional $p$-Laplacian type with $p>1$, which hinders the corresponding statement into the linear case when $p=2$. Therefore, Theorem \ref{main1}  improves \cite[Theorem 1.1]{BK}, \cite[Theorems 1 and 2]{DZZ} and \cite[Theorem 1.1]{St1}. Moreover, in contrast to \cite[Theorems 1 and 2]{DZZ} and \cite[Theorem 1.1]{St1} (see also \cite[Theorems 1.5 and 1.6]{PZ} and the references therein for developments), we are able to control the local supremum of subsolutions in a ball by a nonlocal parabolic tail which uses the $L^1$-norm in time instead of
the $L^\infty$-norm
as in the quoted papers. Besides, we should emphasize that Theorem \ref{main1} deals with the local boundedness of subsolutions to the parabolic equation associated with fractional $p$-Laplacian type operators  for all $p>1$, while \cite[Theorem 1.1]{St1} is only concerned with $p>2$
and \cite[Theorems 1 and 2]{DZZ} is also split into two cases (with different statements) according to $p\in (1,2]$ and $p\ge2$ but there are some restrict conditions in the case that $p\in (1,2)$.
    \end{itemize}
\end{remark}

\subsection{Overview of
the related literature}
Nonlocal elliptic equations of fractional $p$-Laplacian type operators with $p>1$ (i.e., $-\sL_tu=0$) have received great attention in recent years. The existence of weak solutions to the elliptic counterpart of \eqref{EQ} is considered in  \cite{DKP1} under variational framework, and then the local boundedness and the H\"{o}lder continuity of weak solutions are investigated
there
by the De Giorgi-Nash-Moser theory. Based on the boundedness result, a very interesting nonlocal version of the elliptic Harnack inequality for solutions to the equation \eqref{EQ} is established  in \cite{DKP2}. This
nonlocal elliptic Harnack inequality involves the so-called tail of the negative part of the solution $u$, and does not require the solution to be globally positive. For the equation involving a general source term $f(x,u)$, Cozzi \cite{CM} introduced the fractional De Giorgi classes, and proved the local boundedness and the H\"{o}lder continuity as well as the elliptic Harnack inequality of weak solutions to the equation $-\sL_tu=f$.

Regularity theory of weak solutions to the nonlocal parabolic equations of type \eqref{EQ} seems incomplete. V\'{a}zquez \cite{VJL} provided the existence and the uniqueness of strong solutions to \eqref{EQ} under the assumption that $u=0$ in $\R^d\setminus D$, and also investigated large time behaviors of solutions. The
well-posedness
for strong solutions to the equation \eqref{EQ} subject to the Dirichlet condition or the Neumann condition is considered in \cite{MRT}, where asymptotic behaviours of strong solution are also studied. Str\"{o}mqvist \cite{St1} investigated the equation \eqref{EQ} with $u=g$ in $\R^d\setminus D$, and obtained the
well-posedness
and the local boundedness of weak solutions under the assumption that the associated kernel  $J(t;x,y)$ satisfies \eqref{J0} and $p\ge2$.  
Byun and Kim \cite{BK} obtained the local boundedness of subsolutions to the equation \eqref{EQ} in terms of the nonlocal parabolic tail with the $L^q$-norm in time for $q>p-1$. 
For the equation involving a general source term $f(t,x,u)$, Ding et al.\,\cite{DZZ} imposed some structural conditions on the function $f$ such that the local boundedness holds for local weak solutions to $\partial_tu-\sL_tu=f$ with all $p>1$. Furthermore, they also obtained in \cite{DZZ} the H\"{o}lder continuity of bounded solutions when $p>2$. In particular, results in \cite{DZZ}
used the nonlocal supremum parabolic tail. More recently, in the linear case $p=2$, Kassmann and Weidner \cite{KW2} proved the parabolic Harnack inequality and the H\"{o}lder continuity within the variational framework for weak solutions to linear parabolic nonlocal equations, where the nonlocal parabolic tail with the $L^1$-norm in time is adopted.
As far as we know, there is no theory yet for the local boundedness of subsolutions to the nonlinear and nonlocal equation of type \eqref{EQ} in terms of the nonlocal parabolic tail with the $L^q$-norm in time with some proper $q>0$. Thus, the purpose of this work is to fill this gap for weak solutions to \eqref{EQ} with all $p>1$. We believe that the result of our paper would be a basis for further study of regularity theory (including the H\"{o}lder continuity) of weak solutions to the equation of the type \eqref{EQ};
see \cite{CKW1,CKW2, KW2, St2} and the references therein for the linear case $p=2$.

\smallskip

The proof of Theorem \ref{main1} is based on a decomposition of the nonlocal supremum parabolic tail into an intermediate tail and a remaining tail, which is carried out by a
cut-off procedure on the solution. The idea is motivated by the recent work \cite{KW2}, where the local boundedness with the $L^1$-tail was established for weak solutions to linear parabolic nonlocal equations.
In details, for the first part (intermediate tail) in the decomposition, we are able to make use of the equation, which allows us to estimate it by an
$L^{p-1}$-tail (see Proposition \ref{sup}); while for the second part (remaining tail),
via a new
level set truncation in the De Giorgi-Nash-Moser iteration, we can directly profit from a cancellation of the tail terms (see Proposition \ref{Cl}) in the iteration scheme. For this purpose, we need to prove the Caccioppoli-type inequality with the new form, which is proper to
fractional
$p$-Laplacian type operators with all $p>1$, see Lemma \ref{2}.
Moreover, to realize the De Giorgi-Nash-Moser iteration for the proof of Theorem \ref{main1}, we will apply Sobolev inequalities with different type for the fractional $p$-Laplace operator according to $p\in [2,\infty)$ and $p\in (1,2)$. In particular, for $p\in (1,2)$, some efforts including the choices of iteration orders and the constants involved in the argument are required; see Part (2) in the proof of Theorem \ref{main1} for the details.

 \ \

The paper is arranged as follows. In Section \ref{section2} we give some preparations for the proof. In particular, we
first present the Caccioppoli-type inequality for the nonlocal parabolic operator with
fractional
$p$-Laplacian type for all $p>1$, and then establish an inequality for the estimate of the $L^\infty$-tail of subsolutions in terms of the $L^{p-1}$-tail.
Section \ref{section3} is devoted to proving Theorem \ref{main1} with the
the
aid of the De Giorgi-Nash-Moser iteration.
The proof here is split into two cases according to $p\in [2,\infty)$ and $p\in (1,2)$. We also highlight the differences of the proof for these two cases.
\section{Preparations}\label{section2}

To establish the local boundedness of
subsolutions associated with the operator $\sL_t$, we need the following nonlocal Caccioppoli-type inequality.
We should stress that though the proof of the lemma below is partly motivated by \cite[Lemma 3.2]{St1}, some nontrivial modifications are required in order to make sure that it can be efficiently used later for
fractional
 $p$-Laplacian type operators with all $p>1$.

\begin{lemma}\label{2}{\bf (Caccioppoli-type inequality)}\,\,
Let $p\in(1,\infty)$ and
$\xi>0$. Let  $x_0\in \R^d$, $r>0$, $\tau_1<\tau_2$ and $h>0$, and set $B_r:=B_r(x_0)$. Assume that $u$ is a non-negative subsolution on $B_{r}\times(\tau_1-h,\tau_2)$. Let $k$ be a non-negative and differentiable function on $(\tau_1-h,\tau_2)$ such that $k'\ge0$ on $(\tau_1-h,\tau_2)$,
 $v(t,x)=(u(t,x)-k(t))_+$ and $w=v^{(p-1+\xi)/p}$. There exists a constant $C_1>0$ $($which depends on $p$ and $\xi$ only$)$, such that for every non-negative function $\phi(t,x)$ which is compactly supported in $B_r$ for any fixed $t\in(\tau_1-h,\tau_2)$ and is differentiable with respect to $t$ for any fixed $x\in B_r$, it holds that
\begin{align*}
&\int_{\tau_1-h}^{\tau_2}\int_{B_r\times B_r}|w\phi(t,x)-w\phi(t,y)|^p J(t;x,y)\,dx\,dy\,dt\\
&+\int_{\tau_1-h}^{\tau_2}\int_{B_r\times B_r}|w(t,x)-w(t,y)|^p\phi(t,x)^p J(t;x,y)\,dx\,dy\,dt+\frac{1}{1+\xi}\sup_{ 
\tau_1-h<t<\tau_2}\int_{B_{r}}v(t,x)^{1+\xi}\phi(t,x)^p\,dx\\
&\le C_1\int_{\tau_1-h}^{\tau_2}\int_{B_r}v(t,x)^{\xi} \phi(t,x)^p\left(\int_{B_r^c}v(t,y)^{p-1}J(t;x,y)\,dy\right)\,dx\,dt\\
&\quad+C_1\int_{\tau_1-h}^{\tau_2}\int_{B_r\times B_r}\max\{w(t,x),w(t,y )\}^p|\phi(t,x)-\phi(t,y)|^pJ(t;x,y)\,dx\,dy\,dt\\
&\quad+
\frac{2}{(1+\xi)}\int_{\tau_1-h}^{\tau_2}\int_{B_{r}} |\partial_t\phi(t,x)^{p}|v(t,x)^{1+\xi}\,dx\,dt+
\frac{2}{(1+\xi)}\int_{B_{r}}v(\tau_1-h,x)^{1+\xi}\phi(\tau_1-h,x)^p\,dx\\
&\quad-\int_{\tau_1-h}^{\tau_2}\int_{B_{r}} v(t,x)^\xi\phi(t,x)^pk'(t)\,dx\,dt.
\end{align*}
\end{lemma}

\begin{proof}
Let
$q=1-\xi<1$. Let $t_1=\tau_1-h$ and $t_2\in(t_1,\tau_2]$ to be determined later. Using the fact that $u(t,x)$ is a subsolution on $B_{r}\times(\tau_1-h,\tau_2)$ and
multiplying $v(t,x)^{1-q}\phi(t,x)^p$ in both sides of \eqref{PE},  we find
\begin{align*}0
&\ge \int_{t_1}^{t_2}\int_{B_{r}} -\sL_t u(t,x)v(t,x)^{1-q}\phi(t,x)^p\,dx\,dt +\int_{t_1}^{t_2}\int_{B_r}v(t,x)^{1-q}\phi(t,x)^p \partial_tu(t,x)\,dx\,dt\\
&=\frac{1}{2}\int_{t_1}^{t_2}\int_{B_{r}\times B_{r}} |u(t,x)-u(t,y)|^{p-2} (u(t,x)-u(t,y))(v(t,x)^{1-q}\phi(t,x)^p- v(t,y)^{1-q}\phi(t,y)^p) J(t;x,y)\,dx\,dy\,dt\\
&\quad +\int_{t_1}^{t_2}\int_{B_{r}\times B^c_{r}}|u(t,x)-u(t,y)|^{p-2} (u(t,x)-u(t,y))v(t,x)^{1-q}\phi(t,x)^p J(t;x,y)\,dx\,dy\,dt\\
&\quad +\int_{t_1}^{t_2}\int_{B_{r}}v(t,x)^{1-q}\phi(t,x)^p \partial_tu(t,x)\,dx\,dt\\
&=:I_1+I_2+I_3.
\end{align*}

For $I_3$, since
$$v(t,x)^{1-q}\phi(t,x)^p \partial_tu(t,x)=\frac{1}{2-q}\phi(t,x)^p\partial_tv(t,x)^{2-q}+v(t,x)^{1-q}\phi(t,x)^pk'(t),$$
we obtain
\begin{align*}
I_3&=\frac{1}{2-q}\int_{t_1}^{t_2}\int_{B_{r}}\phi(t,x)^p\partial_tv(t,x)^{2-q}\,dx\,dt+\int_{t_1}^{t_2}\int_{B_{r}}
v(t,x)^{1-q}\phi(t,x)^pk'(t)\,dx\,dt\\
&=\frac{1}{2-q}\int_{B_{r}} v(t_2,x)^{2-q}\phi(t_2,x)^p\,dx-\frac{1}{2-q}\int_{B_{r}} v(t_1,x)^{2-q}\phi(t_1,x)^p\,dx\\
&\quad-\frac{1}{2-q}\int_{t_1}^{t_2}\int_{B_{r}}\partial_t\phi(t,x)^{p}v(t,x)^{2-q }\,dx\,dt+\int_{t_1}^{t_2}\int_{B_{r}}
v(t,x)^{1-q}\phi(t,x)^pk'(t)\,dx\,dt.
\end{align*}

For $I_2$, noting that
\begin{align*}|u(t,x)-u(t,y)|^{p-2} (u(t,x)-u(t,y))  v(t,x)^{1-q}\ge
& -(u(t,y)-u(t,x))_+^{p-1}
v(t,x)
^{1-q}\\
\ge &-(u(t,y)-k(t))_+^{p-1}
v(t,x)^{1-q}\\
= &-v(t,y)^{p-1}v(t,x)^{1-q},
\end{align*} we
obtain
\begin{align*}
I_{2}\ge &-\int_{t_1}^{t_2}\int_{B_r}v(t,x)^{1-q} \phi(t,x)^p\left(\int_{B_r^c}v(t,y)^{p-1}J(t;x,y)\,dy\right)\,dx\,dt.
\end{align*}

For $I_{1}$ we consider the integral involved in. Without loss of generality we may and do assume that $u(t,x)\ge u(t,y)$; otherwise we just exchange the roles of $x$ and $y$. It holds that
\begin{align*} &|u(t,x)-u(t,y)|^{p-2} (u(t,x)-u(t,y))( v(t,x)^{1-q} \phi(t,x)^p- v(t,y)^{1-q} \phi(t,y)^p)\\
&=(u(t,x)-u(t,y))^{p-1}( v(t,x)^{1-q} \phi(t,x)^p- v(t,y)^{1-q} \phi(t,y)^p)\\
&=\begin{cases}(v(t,x)-v(t,y))^{p-1}( v(t,x)^{1-q} \phi(t,x)^p- v(t,y)^{1-q} \phi(t,y)^p),&\quad u(t,x),u(t,y)>k(t),\\
(u(t,x)-u(t,y))^{p-1}v(t,x)^{1-q}\phi(t,x)^p,&\quad u(t,x)>k(t)\ge u(t,y),\\
0,&\quad k(t)\ge u(t,x)\end{cases}\\
&\ge (v(t,x)-v(t,y))^{p-1}(v(t,x)^{1-q}\phi(t,x)^p-v(t,y)^{1-q}\phi(t,y)^p).
\end{align*}
Recall from \cite[Lemma 3.1]{DKP1} that for any $p>1$ and $\varepsilon\in (0,1]$,
$$\phi(t,y)^p\le \phi(t,x)^p+c_p\varepsilon\phi(t,x)^p+(1+c_p\varepsilon)\varepsilon^{1-p}|\phi(t,x)-\phi(t,y)|^p$$
holds
with $c_p=(p-1)\Gamma(1\vee(p-2))$. Set
$$\varepsilon=\delta\frac{v(t,x)-v(t,y)}{v(t,x)},$$where $\delta\in(0,1)$ will be chosen later.
Then,
\begin{align*}
&(v(t,x)-v(t,y))^{p-1}\left(\frac{\phi(t,x)^p}{v(t,x)^{q-1}}-\frac{\phi(t,y)^p}{v(t,y)^{q-1}}\right)\\
&=(v(t,x)-v(t,y))^{p-1}\left(\frac{\phi(t,x)^p}{v(t,x)^{q-1}}-\frac{\phi(t,x)^p}{v(t,y)^{q-1}}+\frac{\phi(t,x)^p-\phi(t,y)^p}{v(t,y)^{q-1}}\right)\\
&\geq(v(t,x)-v(t,y))^{p-1}\left(\frac{\phi(t,x)^p}{v(t,x)^{q-1}}-\frac{\phi(t,x)^p}{v(t,y)^{q-1}}\left(1+c_p\delta\frac{v(t,x)-v(t,y)}{v(t,x)}\right)\right)\\
&\quad-\frac{(v(t,x)-v(t,y))^{p-1}}{v(t,y)^{q-1}}\left(\left(1+c_p\delta\frac{v(t,x)-v(t,y)}{v(t,x)}\right)\delta^{1-p}\frac{(v(t,x)-v(t,y))^{1-p}}{v(t,x)^{1-p}}\right)\\
&\quad\quad\times|\phi(t,x)-\phi(t,y)|^p\\
&=:A+B.
\end{align*}

We further estimate $A$ as follows
\begin{equation}\label{V}\begin{split}
&A=(v(t,x)-v(t,y))^{p-1}\frac{\phi(t,x)^p}{v(t,y)^{q-1}}\left(\frac{v(t,y)^{q-1}}{v(t,x)^{q-1}}-1-c_p\delta\frac{v(t,x)-v(t,y)}{v(t,x)}\right)\\
&\quad=(v(t,x)-v(t,y))^{p-1}\frac{\phi(t,x)^p}{v(t,y)^{q}}(v(t,x)-v(t,y))\\
&\qquad\times\left(\frac{v(t,y)^{q}}{v(t,x)^{q-1}(v(t,x)-v(t,y))}-\frac{v(t,y)}{v(t,x)-v(t,y)}-c_p\delta\frac{v(t,y)}{v(t,x)}\right)\\
&\quad=(v(t,x)-v(t,y))^p\frac{\phi(t,x)^p}{v(t,y)^{q}}\left(g\left(\frac{v(t,x)}{v(t,y)}\right)-c_p\delta\frac{v(t,y)}{v(t,x)}\right),
\end{split}\end{equation}
where
$$g(a)=\frac{a^{1-q}-1}{a-1},\quad a\ge 1.$$
Note that $\xi>0$, so $q<1$. Below we will split
$v(t,x)\ge v(t,y)$ into the following two cases.

\textbf {Case 1: $v(t,x)\ge2v(t,y)$}. If $a\ge2$, then
\begin{equation}\label{B}
g(a)=\frac{a^{1-q}-1}{a-1}\ge\frac{(1-2^{q-1})a^{1-q}}{a-1},
\end{equation} thanks to the fact that $1-q>0$.
Thus, according to \eqref{V} and \eqref{B},
\begin{align*}
&A\ge(v(t,x)-v(t,y))^p\frac{\phi(t,x)^p}{v(t,y)^{q}}\left(\frac{(1-2^{q-1})\frac{v(t,y)^q}{v(t,x)^{q-1}}}{v(t,x)-v(t,y)}-c_p\delta\frac{v(t,y)}{v(t,x)}\right)\\
&\quad=\frac{(v(t,x)-v(t,y))^{p-1}}{v(t,x)^{q-1}}\phi(t,x)^p\left(1-2^{q-1}-c_p\delta\frac{v(t,y)}{v(t,x)}\frac{v(t,x)-v(t,y)}{v(t,y)^q}v(t,x)^{q-1}\right).\\
\end{align*}
Furthermore, note that, due to $v(t,x)\ge2v(t,y)$ and $q<1$,
$$-\frac{v(t,y)}{v(t,x)}\frac{v(t,x)-v(t,y)}{v(t,y)^q}v(t,x)^{q-1}\ge -\left(\frac{v(t,y)}{v(t,x)}\right)^{1-q}\ge-\left(\frac{1}{2}\right)^{1-q}=-2^{q-1}.$$
Thus
\begin{align*}A\ge&\frac{(v(t,x)-v(t,y))^{p-1}}{v(t,x)^{q-1}}\phi(t,x)^p\left(1-2^{q-1}-c_p\delta2^{q-1}\right)\\
\ge&2^{1-p}v(t,x)^{p-q}\phi(t,x)^p\left(1-2^{q-1}-c_p\delta2^{q-1}\right)\\
\ge&2^{1-p}\left(v(t,x)^{\frac{p-q}{p}}-v(t,y)^{\frac{p-q}{p}}\right)^p\phi(t,x)^p\left(1-2^{q-1}-c_p\delta2^{q-1}\right).\end{align*}
Take $\delta>0$ small enough so that $1-2^{q-1}\ge 2c_p\delta2^{q-1}$, we
obtain
\begin{align*}
&A\ge2^{-p}(1-2^{q-1})\left(v(t,x)^{\frac{p-q}{p}}-v(t,y)^{\frac{p-q}{p}}\right)^p\phi(t,x)^p\\
&\quad=2^{-p}(1-2^{q-1})(w(t,x)-w(t,y))^p\phi(t,x)^p,
\end{align*} where
$w=v^{(p-1+\xi)/p}=v^{(p-q)/p}$.

\textbf {Case 2: $v(t,y)\le v(t,x)<2v(t,y)$}. If $1\le a<2$, then
$$(a^{1-q}-1)(a^q+a^{q-(1-q)}+\cdots+a^{q-(m-1)(1-q)}+1)\ge a-1.$$ where
$m:=m(q)$ is the minimum of the natural number such that $m(1-q)\ge q.$ Note that $m$ only depends on $q$ (and so on $\xi$). The inequality above along with $1-q>0$ and $m\ge1$ implies that for all $1\le a<2$,
\begin{align*}g(a)=\frac{a^{1-q}-1}{ a-1}\ge& \frac{1}{a^q+a^{q-(1-q)}+\cdots+a^{q-(m-1)(1-q)}+1}\ge \frac{1}{ma^q+1}\\
\ge&\begin{cases}\frac{1}{m+1}&\quad q<0\\
\frac{1}{m2^q+1}&\quad 0\le q<1\end{cases}\\
\ge&\frac{1}{m2^{q\vee0}+1}.
\end{align*}
Take $\delta>0$ small enough so that $\frac{1}{m2^{q\vee0}+1}\ge 2c_p \delta $. By \eqref{V}, the fact that $g(a)\ge\frac{1}{m2^{q\vee0}+1}$ for all $1\le a<2$ and the choice of $\delta$, we have
\begin{align*}A\ge&\frac{\phi(t,x)^p}{v(t,y)^q}(v(t,x)-v(t,y))^p\left(\frac{1}{m2^{q\vee0}+1}-
c_p \delta\right)\\
\ge&\frac{1}{2(m2^{q\vee0}+1)}\frac{\phi(t,x)^p}{v(t,y)^q}(v(t,x)-v(t,y))^p.\end{align*}
On the other hand,
recalling that $w=v^{(p-1+\xi)/p}=v^{(p-q)/p}$,
\begin{align*}
(w(t,x)-w(t,y))^p&=
\left(\frac{p}{p-q}\right)^{-p}
\left(\int^{v(t,x)}_{v(t,y)}\tau^{-q/p}d\tau\right)^p\\
&\le\begin{cases}\left(\frac{p}{p-q}
\right)^{-p}
\frac{(v(t,x)-v(t,y))^p}{v(t,x)^q}\le \left(\frac{p}{p-q}
\right)^{-p}
2^{-q}\frac{(v(t,x)-v(t,y))^p}{v(t,y)^q}&\quad q<0\\
\left(\frac{p}{p-q}
\right)^{-p}
\frac{(v(t,x)-v(t,y))^p}{v(t,y)^q}&\quad 0\le q<1\end{cases}\\
&\le \left(\frac{p}{p-q}
\right)^{-p}
(2^{-q}\vee1)\frac{(v(t,x)-v(t,y))^p}{v(t,y)^q},
\end{align*} where in the first inequality we used the fact that $v(t,y)\le v(t,x)<2v(t,y)$
which is the
assumption in this case.
Thus,
\begin{align*}
A&\ge\frac{1}{2(m2^{q\vee0}+1)}\left(\frac{p}{p-q}
\right)^{p}
(2^{q}\wedge1)\phi(t,x)^p(w(t,x)-w(t,y))^p.
\end{align*}
Combining Case 1 with Case 2,
we show that there are $\delta>0$ small enough and $c_1>0$ so that
\begin{equation}\label{N}
A\ge c_1(w(t,x)-w(t,y))^p\phi(t,x)^p.
\end{equation}

According to the facts that
$$-v(t,y)^{1-q}\ge-v(t,x)^{1-q}, \quad (v(t,x)-v(t,y))/v(t,x)\le1,$$
we find
\begin{equation}\label{M}
B\ge-c_2w(t,x)^p|\phi(t,x)-\phi(t,y)|^p.
\end{equation}

Finally, putting both estimates for $A$ and $B$ together, we find that for $v(t,x)\ge v(t,y)$,
$$A+B\ge c_1(w(t,x)-w(t,y))^p\phi(t,x)^p-c_2w(t,x)^p|\phi(t,x)-\phi(t,y)|^p.$$
By applying the
elementary
inequality $(a+b)^{p}\le 2^{p-1}(a^p+b^p)$ for all $a,b\ge0$,
\begin{align*}
&|(w(t,x)\phi(t,x)-w(t,y)\phi(t,y)|^p-2^{p-1}\max\{w(t,x),w(t,y )\}^p|\phi(t,x)-\phi(t,y)|^p\\
&\le2^{p-1}|w(t,x)-w(t,y)|^p\phi(t,x)^p.
\end{align*}
Therefore,
there are positive constants $c_3$ and $c_4$ such that
\begin{equation}\label{F}\begin{split}
I_1\ge
&c_3\int_{t_1}^{t_2}\int_{B_r\times B_r}|w\phi(t,x)-w\phi(t,y)|^p J(t;x,y)\,dx\,dy\,dt\\
&+c_3\int_{t_1}^{t_2}\int_{B_r\times B_r}|w(t,x)-w(t,y)|^p\phi(t,x)^p J(t;x,y)\,dx\,dy\,dt\\
&-c_4\int_{t_1}^{t_2}\int_{B_r\times B_r}\max\{w(t,x),w(t,y )\}^p|\phi(t,x)-\phi(t,y)|^pJ(t;x,y)\,dx\,dy\,dt.
\end{split}\end{equation}

According to all the estimates above, we find that
\begin{align*}
&\int_{t_1}^{t_2}\int_{B_r\times B_r}|w\phi(t,x)-w\phi(t,y)|^p J(t;x,y)\,dx\,dy\,dt\\
&+\int_{t_1}^{t_2}\int_{B_r\times B_r}|w(t,x)-w(t,y)|^p\phi(t,x)^p J(t;x,y)\,dx\,dy\,dt+\frac{1}{2-q}\int_{B_{r}} v(t_2,x)^{2-q}\phi(t_2,x)^p\,dx\\
&\le c_5\int_{t_1}^{t_2}\int_{B_r}v(t,x)^{1-q} \phi(t,x)^p\left(\int_{B_r^c}v(t,y)^{p-1}J(t;x,y)\,dy\right)\,dx\,dt\\
&\quad+c_5\int_{t_1}^{t_2}\int_{B_r\times B_r}\max\{w(t,x),w(t,y )\}^p|\phi(t,x)-\phi(t,y)|^pJ(t;x,y)\,dx\,dy\,dt\\
&\quad+\frac{1}{2-q}\int_{t_1}^{t_2}\int_{B_{r}}|\partial_t\phi(t,x)^{p}|v(t,x)^{2-q }\,dx\,dt+\frac{1}{2-q}\int_{B_{r}} v(t_1,x)^{2-q}\phi(t_1,x)^p\,dx\\
&\quad-\int_{t_1}^{t_2}\int_{B_{r}}
v(t,x)^{1-q}\phi(t,x)^pk'(t)\,dx\,dt.
\end{align*}
That is,
\begin{align*}&\int_{t_1}^{t_2}\int_{B_r\times B_r}|w\phi(t,x)-w\phi(t,y)|^p J(t;x,y)\,dx\,dy\,dt\\
&+\int_{t_1}^{t_2}\int_{B_r\times B_r}|w(t,x)-w(t,y)|^p\phi(t,x)^p J(t;x,y)\,dx\,dy\,dt\\
&+\int_{t_1}^{t_2}\int_{B_{r}}
v(t,x)^{1-q}\phi(t,x)^pk'(t)\,dx\,dt+\frac{1}{2-q}\int_{B_{r}} v(t_2,x)^{2-q}\phi(t_2,x)^p\,dx\\
&\le c_5\int_{t_1}^{t_2}\int_{B_r}v(t,x)^{1-q} \phi(t,x)^p\left(\int_{B_r^c}v(t,y)^{p-1}J(t;x,y)\,dy\right)\,dx\,dt\\
&\quad+c_5\int_{t_1}^{t_2}\int_{B_r\times B_r}\max\{w(t,x),w(t,y )\}^p|\phi(t,x)-\phi(t,y)|^pJ(t;x,y)\,dx\,dy\,dt\\
&\quad+\frac{1}{2-q}\int_{t_1}^{t_2}\int_{B_{r}}|\partial_t\phi(t,x)^{p}|v(t,x)^{2-q }\,dx\,dt+\frac{1}{2-q}\int_{B_{r}} v(t_1,x)^{2-q}\phi(t_1,x)^p\,dx.\end{align*}
Since $k'(t)\ge0$ for all $t\in (t_1,t_2)$, all the terms in the inequality are non-negative. Hence,
\begin{align*}&\int_{t_1}^{\tau_2}\int_{B_r\times B_r}|w\phi(t,x)-w\phi(t,y)|^p J(t;x,y)\,dx\,dy\,dt\\
&+\int_{t_1}^{\tau_2}\int_{B_r\times B_r}|w(t,x)-w(t,y)|^p\phi(t,x)^p J(t;x,y)\,dx\,dy\,dt\\
&+\int_{t_1}^{\tau_2}\int_{B_{r}}
v(t,x)^{1-q}\phi(t,x)^pk'(t)\,dx\,dt+\sup_{t_2\in ( 
t_1,\tau_2)}\frac{1}{2-q}\int_{B_{r}} v(t_2,x)^{2-q}\phi(t_2,x)^p\,dx\\
&\le 2c_5\int_{t_1}^{\tau_2}
\int_{B_r}v(t,x)^{1-q} \phi(t,x)^p\left(\int_{B_r^c}v(t,y)^{p-1}J(t;x,y)\,dy\right)\,dx\,dt\\
&\quad+2c_5\int_{t_1}^{\tau_2}
\int_{B_r\times B_r}\max\{w(t,x),w(t,y )\}^p|\phi(t,x)-\phi(t,y)|^pJ(t;x,y)\,dx\,dy\,dt\\
&\quad+\frac{2}{2-q}\int_{t_1}^{\tau_2}
\int_{B_{r}}|\partial_t\phi(t,x)^{p}|v(t,x)^{2-q }\,dx\,dt+\frac{2}{2-q}\int_{B_{r}} v(t_1,x)^{2-q}\phi(t_1,x)^p\,dx.\end{align*} The proof is complete.
\end{proof}

Before moving further, we need two lemmas.
First, we give an extension of \cite[Lemma 1.1]{GG}.
\begin{lemma}\label{EX}
Assume that there are $A_1$, $A_2$, $B$, $\gamma_1$, $\gamma_2>0$ and $\kappa\in(0,1)$ such that for any $R>0$ and any bounded function $f : [R/2,R]\rightarrow[0,\infty)$,
$$f(r)\le A_1(s-r)^{-\gamma_1}+A_2(s-r)^{-\gamma_2}+B+\kappa f(s),\quad R/2\le r\le s\le R.$$
Then, there exists a constant $c>0$ $($that is independent of $A_1$, $A_2$ and $B)$ such that
$$f(R/2)\le c\left(R^{-\gamma_1}A_1+R^{-\gamma_2}A_2\right)+cB.$$
\end{lemma}
\begin{proof}
Set
$$t_0=R/2; \quad t_{i+1}-t_i=(1-\tau)\tau^iR/2,\quad i\ge0$$
with $0<\tau<1$. By the iteration, for all $n\ge1$,
$$f(t_0)\le\kappa^nf(t_n)+\left[\frac{A_1}{(1-\tau)^{\gamma_1}}\left(\frac{R}{2}\right)^{-\gamma_1}+\frac{A_2}{(1-\tau)^{\gamma_2}}\left(\frac{R}{2}\right)^{-\gamma_2}+B\right]\sum^{n-1}_{i=0}\kappa^i\tau^{-i(\gamma_1+\gamma_2)}.$$
We choose now $\tau\in (0,1)$ such that $\tau^{-(\gamma_1+\gamma_2)}\kappa<1$ and let $n\rightarrow\infty$, getting
$$f(R/2)\le c\left(R^{-\gamma_1}A_1+R^{-\gamma_2}A_2\right)+cB$$
with $c=\left(\frac{2}{1-\tau}\right)^{(\gamma_1\vee\gamma_2)}\left(1-\tau^{-(\gamma_1+\gamma_2)}\kappa\right)^{-1}$. The proof is complete.
\end{proof}

We next
present the following elementary iteration lemma, see, e.g., \cite[Lemma 7.1]{Gi}.
\begin{lemma}\label{iteration}
Let $\{Y_j\}_{j\ge0}$ be a sequence of real positive numbers satisfying
$$Y_{j+1}\le c_0b^jY_j^{1+\beta},\quad j\ge0 $$
for some constants $c_0>0$, $b>1$ and $\beta>0$. If
$$Y_0\le c_0^{-1/\beta}b^{-1/\beta^2},$$
then
\begin{equation}\label{a-}
Y_j\le b^{-j/\beta}Y_0,\quad j\ge0,
\end{equation}
which in particular yields $\lim_{j\rightarrow\infty}Y_j=0$.
\end{lemma}

\begin{proof}
We proceed by the induction. The inequality \eqref{a-} is obviously true for $j=0$.
Assume now that \eqref{a-} holds for $j\ge1$. Then,
$$Y_{j+1}\le c_0b^jb^{-j(1+\beta)/\beta}Y_0^{1+\beta}=(c_0b^{1/\beta}Y_0^\beta)b^{-(j+1)/\beta}Y_0\le b^{-(j+1)/\beta}Y_0,$$
so \eqref{a-} holds for $j+1$.
\end{proof}

In the proof of Theorem \ref{main1}, we also need the following lemma (see \cite[Lemmas 2.3 and 2.4]{DZZ}).
\begin{lemma}\label{L:lemma3.3} Let $s\in (0,1)$ and $p>1$. Then there is a constant $c_1>0$ such that for all $t_1<t_2$, $x_0\in \R^d$,  $r>0$ and $f\in L^p(t_1,t_2;W^{s,p}(B_r(x_0)))$,
\begin{itemize}
\item[{\rm(i)}] \begin{align*}&\int_{t_1}^{t_2}\oint_{B_r(x_0)}|f(t,x)|^{p(1+2s/d)}\,dx\,dt\\
&\le c_1\left(r^{sp}\int_{t_1}^{t_2}\int_{B_r(x_0)}\oint_{B_r(x_0)} \frac{|f(t,x)-f(t,y)|^p}{|x-y|^{d+sp}}\,dx\,dy\,dt+\int_{t_1}^{t_2}\oint_{B_r(x_0)}|f(t,x)|^p\,dx\,dt\right)\\
&\quad\quad\times \left(\sup_{t_1\le t\le t_2}\oint_{B_r(x_0)}|f(t,x)|^2\,dx\right)^{sp/d}.\end{align*}
\item[{\rm(ii)}] \begin{align*}&\int_{t_1}^{t_2}\oint_{B_r(x_0)}|f(t,x)|^{p(1+ps/d)}\,dx\,dt\\
&\le c_1\left(r^{sp}\int_{t_1}^{t_2}\int_{B_r(x_0)}\oint_{B_r(x_0)} \frac{|f(t,x)-f(t,y)|^p}{|x-y|^{d+sp}}\,dx\,dy\,dt+\int_{t_1}^{t_2}\oint_{B_r(x_0)}|f(t,x)|^p\,dx\,dt\right)\\
&\quad\quad\times \left(\sup_{t_1\le t\le t_2}\oint_{B_r(x_0)}|f(t,x)|^p\,dx\right)^{sp/d}.\end{align*}
\end{itemize}\end{lemma}

Recall that $$I^\ominus_r(t_0)=(t_0-r^{sp},t_0),\quad  B_r(x_0)=\{x\in \R^d:|x-x_0|<r\}
.$$
\begin{proposition}\label{sup}
Fix $p>1$. Let $u$ be a subsolution to $\partial_tu-\sL_tu=0$ in $I^\ominus_R(t_0)\times B_R(x_0)$ with
$t_0\in \R$, $x_0\in \R^d$ and $R>0$. Let $v(t,x)=(u(t,x)-k)_+$ with $k\ge0$. Then, there is a constant $C_2>0$ $($which depends on $p$ only$)$ so that
\begin{align*}
\sup_{I^\ominus_{R/2}(t_0)}\left(\oint_{B_{R/2}(x_0)}v(t,x)^p\,dx\right)^{\frac{1}{p}}
\le& C_2\left(\oint_{I^\ominus_R(t_0)}\oint_{B_R(x_0)}v(t,x)^{2p-2}\,dx\,dt+\oint_{I^\ominus_R(t_0)}\oint_{B_R(x_0)}v(t,x)^p\,dx\,dt\right)^{\frac{1}{p}}\\
&+C_2{\rm Tail}_{p-1}^{p-1}(v; x_0,R/2, t_0-R^{sp}, t_0).
\end{align*}
 \end{proposition}

\begin{proof}
Let $R/2\le r\le s\le R$. Choose $\phi(t,x)=\phi_1(x)\phi_2(t)$ such that $\phi_1
:B_s(x_0)\to [0,1]$ is a smooth
cut-off function of $B_r(x_0)\subset B_{\frac{r+s}{2}}(x_0)$, that is, $\phi_1(x)=1$ for all $x\in B_r(x_0)$ and $\phi_1(x)=0$ for all $x\in B_{\frac{r+s}{2}}(x_0)^c$, and $\phi_2: (t_0- s^{sp}, t_0) \to [0,1]$ is differentiable such that $\phi_2(t)=1$ for all $t\in (t_0-r^{sp},t_0)$ and $\phi_2(t)=0$ for all $t\in (t_0- s^{sp}, t_0-(\frac{r+s}{2})^{sp})$.
In particular,
$$|D\phi|\le c_0\left(\frac{s-r}{2}\right)^{-1}
\quad \text{and}
\quad|\partial_t\phi|\le c_0\left[\left(\frac{r+s}{2}\right)^{sp}-r^{sp}\right]^{-1}.$$
Let $k(t)=k\ge0$ be independent of $t$. In particular, $k'(t)=0$. Taking $\tau_2=t_0$, $\tau_1=t_0-r^{sp}$, $h=(\frac{r+s}{2})^{sp}
-r^{sp}$
and $\xi=p-1>0$ as well as the functions $\phi(t,x)$ and $k(t)$ above in the Caccioppoli-type inequality established in Lemma \ref{2}, we find
\begin{align*}
&\sup_{I^\ominus_r(t_0)}\oint_{B_r(x_0)}v(t,x)^p \,dx\\
&\le c_1\sup_{I^\ominus_ 
{\frac{r+s}{2}}(t_0)} \oint_{B_s(x_0)}v(t,x)^p\phi(t,x)^p\,dx\\
&\le
c_2\int_{I^\ominus_{\frac{r+s}{2}}(t_0)}\left(\sup_{x\in B_{\frac{r+s}{2}}(x_0)}\int_{B_s(x_0)^c}v(t,y)^{p-1}J(t;x,y)\,dy\right)\left(\oint_{B_s(x_0)}v(t,x)^{p-1} \phi(t,x)^p\,dx\right)\,dt\\
&\quad+c_2\int_{I^\ominus_{\frac{r+s}{2}}(t_0)}\int_{B_s(x_0)}\oint_{B_s(x_0)} \max\{w(t,x),w(t,y )\}^p|\phi(t,x)-\phi(t,y)|^pJ(t;x,y)\,dx\,dy\,dt\\
&\quad+c_2\int_{I^\ominus_{\frac{r+s}{2}}(t_0)}\oint_{B_s(x_0)}|\partial_t\phi(t,x)^{p}|v(t,x)^{1+\xi}\,dx\,dt\\
&=:
H_1+H_2+H_3,
\end{align*}
where $w=v^{2(p-1)/p}$.

By applying the Young inequality $$ab\le\frac{a^{p/(p-1)}}{p/(p-1)}+
\frac{b^p}{p},
\quad a,b\ge0,$$ we get
\begin{align*}
&
H_1\le c_3\left(\sup_{t\in I^\ominus_{\frac{r+s}{2}}(t_0)}\oint_{B_{\frac{r+s}{2}}(x_0)}v(t,x)^{p-1}dx\right)\left(\left(\frac{s}{s-r}\right)^{d+sp}\int_{I^\ominus_{\frac{r+s}{2}}(t_0)}\int_{B_s(x_0)^c}\frac{v(t,y)^{p-1}}{|x_0-y|^{d+sp}}\,dy\,dt\right)\\
&\le \frac{p-1}{p}\sup_{t\in I^\ominus_{\frac{r+s}{2}}(t_0)}\left(\oint_{B_{\frac{r+s}{2}}(x_0)}v(t,x)^{p-1}dx\right)^{\frac{p}{p-1}}+\frac{c^p_3}{p}\left(\left(\frac{s}{s-r}\right)^{d+sp}\int_{I^\ominus_{\frac{r+s}{2}}(t_0)}\int_{B_s(x_0)^c}\frac{v(t,y)^{p-1}}{|x_0-y|^{d+sp}}\,dy\,dt\right)^p\\
&\le\frac{p-1}{p}\sup_{I^\ominus_{\frac{r+s}{2}}(t_0)}\oint_{B_{\frac{r+s}{2}}(x_0)}v(t,x)^p\,dx+
\frac{c^p_3}{p}(s-r)^{-p(d+sp)}s^{p(d+sp)}\left(\int_{I^\ominus_{\frac{r+s}{2}}(t_0)}\int_{B_s(x_0)^c}\frac{v(t,y)^{p-1}}{|x_0-y|^{d+sp}}\,dy\,dt\right)^p,
\end{align*}
where in the first inequality we used the fact that, for all $x\in B_{\frac{r+s}{2}}(x_0)$ and $y\in B_s(x_0)^c$,
$$\frac{|x_0-y|}{|x-y|}\le\frac{|x_0-x|+|x-y|}{|x-y|}\le1+\frac{s+r}{s-r}=\frac{2s}{s-r},$$ and the last inequality follows from the
Jensen inequality.

Note that
\begin{align*}
&\int_{B_s(x_0)}\oint_{B_s(x_0)} \max\{w(t,x),w(t,y )\}^p|\phi(t,x)-\phi(t,y)|^pJ(t;x,y)\,dx\,dy\\
&\le2\int_{B_s(x_0)}\oint_{B_s(x_0)}w(t,x)^p|\phi(t,x)-\phi(t,y)|^pJ(t;x,y)\,dx\,dy\\
&\le c_4\Bigg[\oint_{B_s(x_0)}w(t,x)^p \left(\frac{s-r}{2}\right)^{-p}\int_{\{|x-y|\le \frac{s-r}{2}\}}|x-y|^p J(t;x,y)\,dy\,dx\\
&\qquad\quad+\oint_{B_s(x_0)} w(t,x)^p\int_{\{|x-y|> \frac{s-r}{2}\}}  J(t;x,y)\,dy\,dx\Bigg]\\
&\le c_5\left(\frac{s-r}{2}\right)^{-sp}\oint_{B_s(x_0)}w(t,x)^p \,dx,
\end{align*} so
$$
H_2\le c_6\left(\frac{s-r}{2}\right)^{-sp}\int_{I^\ominus_{\frac{r+s}{2}}(t_0)}\oint_{B_s(x_0)}v(t,x)^{2p-2} \,dx\,dt.$$

Furthermore, it is obvious that
\begin{align*}
&
H_3\le c_7\left[\left(\frac{r+s}{2}\right)^{sp}-r^{sp}\right]^{-1}\int_{I^\ominus_{\frac{r+s}{2}}(t_0)}\oint_{B_s(x_0)}v(t,x)^p\,dx\,dt.
\end{align*}

Putting all the estimates above, we find that
\begin{align*}
&\sup_{I^\ominus_r(t_0)}\oint_{B_r(x_0)}v(t,x)^p\,dx\\ &\le\frac{p-1}{p}\sup_{I^\ominus_{\frac{r+s}{2}}(t_0)}\oint_{B_{\frac{r+s}{2}}(x_0)}v(t,x)^p\,dx+
\frac{c^p_3}{p}(s-r)^{-p(d+sp)}s^{p(d+sp)}\left(\int_{I^\ominus_{\frac{r+s}{2}}(t_0)}\int_{B_s(x_0)^c}\frac{v(t,y)^{p-1}}{|x_0-y|^{d+sp}}\,dy\,dt\right)^p\\
&\quad+c_6\left(\frac{s-r}{2}\right)^{-sp}\int_{I^\ominus_{\frac{r+s}{2}}(t_0)}\oint_{B_s(x_0)}v(t,x)^{2p-2} \,dx\,dt\\
&\quad+c_7\left[\left(\frac{r+s}{2}\right)^{sp}-r^{sp}\right]^{-1}\int_{I^\ominus_{\frac{r+s}{2}}(t_0)}\oint_{B_s(x_0)}v(t,x)^p\,dx\,dt.\\
&\le\frac{p-1}{p}\sup_{I^\ominus_{\frac{r+s}{2}}(t_0)}\oint_{B_{\frac{r+s}{2}}(x_0)}v(t,x)^p\,dx+c_8(s-r)^{-p(d+sp)}R^{p(d+sp)}\left(\int_{I^\ominus_{R}(t_0)}\int_{B_{R/2}(x_0)^c}\frac{v(t,y)^{p-1}}{|x_0-y|^{d+sp}}\,dy\,dt\right)^p\\
&\quad+c_8\left(\left(\frac{s-r}{2}\right)^{-sp}\vee\left[\left(\frac{r+s}{2}\right)^{sp}-r^{sp}\right]^{-1}\right)\int_{I^\ominus_R(t_0)}\oint_{B_R(x_0)}v(t,x)^{2p-2}\,dx\,dt\\
&\quad+c_8\left(\left(\frac{s-r}{2}\right)^{-sp}\vee\left[\left(\frac{r+s}{2}\right)^{sp}-r^{sp}\right]^{-1}\right)\int_{I^\ominus_R(t_0)}\oint_{B_R(x_0)}v(t,x)^p\,dx\,dt.
\end{align*}

\textbf {Case 1: $sp>1$.}\, In this case,  it follows from the mean value theorem that $$\left[\left(\frac{r+s}{2}\right)^{sp}-r^{sp}\right]^{-1}\le\frac{2^{sp}}{sp}R^{1-sp}(s-r)^{-1}\le 2^{sp} R^{1-sp}(s-r)^{-1},\quad R/2\le r\le s \le R.$$ This along with the fact $$ \left(\frac{s-r}{2}\right)^{-sp}\ge  \left(\frac{R}{2}\right)^{-sp+1} \left(\frac{s-r}{2}\right)^{-1}=2^{sp} R^{1-sp}(s-r)^{-1}$$ yields that

$$\left(\frac{s-r}{2}\right)^{-sp}\vee\left[\left(\frac{r+s}{2}\right)^{sp}-r^{sp}\right]^{-1}\le  \left(\frac{s-r}{2}\right)^{-sp} .$$

Set
$$f(r)=\sup_{I^\ominus_r(t_0)}\oint_{B_r(x_0)}v(t,x)^p\,dx,\quad R/2\le r\le R,$$
$$B=c_8R^{p(d+sp)}\left(\int_{I^\ominus_{R}(t_0)}\int_{B_{R/2}^c(x_0)}\frac{v(t,y)^{p-1}}{|x_0-y|^{d+sp}}\,dy\,dt\right)^p,$$
$$A_1= c_8 \int_{I^\ominus_R(t_0)}\oint_{B_R(x_0)}v(t,x)^{2p-2}\,dx\,dt,\quad A_2= c_8 \int_{I^\ominus_R(t_0)}\oint_{B_R(x_0)}v(t,x)^p\,dx\,dt.$$
Then,
$$f(r)\le(s-r)^{-sp}(A_1+A_2)+(s-r)^{-p(d+sp)}B+\frac{p-1}{p}f((r+s)/2),\quad R/2\le r\le (r+s)/2\le R.$$
Hence, by
Lemma \ref{EX},
$$f(R/2)\le c_9R^{-sp}(A_1+A_2)+c_9R^{-p(d+sp)}B,$$
which yields, noting $|I^\ominus_{R}(t_0)|=R^{sp}$,
\begin{align*}
\sup_{I^\ominus_{R/2}(t_0)}\oint_{B_{R/2}(x_0)}v(t,x)^p\,dx
\le &c_{10}\left(\oint_{I^\ominus_R(t_0)}\oint_{B_R(x_0)}v(t,x)^{2p-2}\,dx\,dt+\oint_{I^\ominus_R(t_0)}\oint_{B_R(x_0)}v(t,x)^p\,dx\,dt\right)\\
& +c_{10}\left(\int_{I^\ominus_{R}(t_0)}\int_{B_{R/2}(x_0)^c}\frac{v(t,y)^{p-1}}{|x_0-y|^{d+sp}}\,dy\,dt\right)^p.
\end{align*}

\textbf {Case 2: $sp\in(0,1]$.} Now, it holds that $$\left[\left(\frac{r+s}{2}\right)^{sp}-r^{sp}\right]^{-1}\ge \left(\frac{s-r}{2}\right)^{-sp}$$
and,  by the mean value theorem, for any $R/2\le r\le s\le R$, $$\left[\left(\frac{r+s}{2}\right)^{sp}-r^{sp}\right]^{-1}\le \frac{2}{sp}R^{1-sp}(s-r)^{-1}.$$
In particular,
$$\left(\frac{s-r}{2}\right)^{-sp}\vee\left[\left(\frac{r+s}{2}\right)^{sp}-r^{sp}\right]^{-1}\le \frac{2}{sp}R^{1-sp}(s-r)^{-1}.$$

Let $f(r)$ and $B$ be
in Case 1, and set
$$A_1=\frac{2c_8}{sp}R^{1-sp}\int_{I^\ominus_R(t_0)}\oint_{B_R(x_0)}v(t,x)^{2p-2}\,dx\,dt,
\quad A_2=\frac{2c_8}{sp}R^{1-sp}\int_{I^\ominus_R(t_0)}\oint_{B_R(x_0)}v(t,x)^p\,dx\,dt.$$
Then,
$$f(r)\le(s-r)^{-1}(A_1+A_2)+(s-r)^{-p(d+sp)}B+\frac{p}{p-1}f((r+s)/2),\quad R/2\le r\le (r+s)/2\le R.$$
According to Lemma \ref{EX}, we have
$$f(R/2)\le c_{11}R^{-1}(A_1+A_2)+c_{11}R^{-p(d+sp)}B,$$
which yields, noting $|I^\ominus_{R}(t_0)|=R^{sp}$,
\begin{align*}
&\sup_{I^\ominus_{R/2}(t_0)}\oint_{B_{R/2}(x_0)}v(t,x)^p\,dx\\
&\le c_{12}R^{-1}\left(R^{1-sp}\int_{I^\ominus_R(t_0)}\oint_{B_R(x_0)}v(t,x)^{2p-2}\,dx\,dt+R^{1-sp}\int_{I^\ominus_R(t_0)}\oint_{B_R(x_0)}v(t,x)^p\,dx\,dt\right)\\
&\quad\,\,+c_{12}\left(\int_{I^\ominus_{R}(t_0)}\int_{B_{R/2}(x_0)^c}\frac{v(t,y)^{p-1}}{|x_0-y|^{d+sp}}\,dy\,dt\right)^p\\
&=c_{12}\left(\oint_{I^\ominus_R(t_0)}\oint_{B_R(x_0)}v(t,x)^{2p-2}\,dx\,dt+\oint_{I^\ominus_R(t_0)}\oint_{B_R(x_0)}v(t,x)^p\,dx\,dt\right)\\
&\quad\,\,+c_{12}\left(\int_{I^\ominus_{R}(t_0)}\int_{B_{R/2}(x_0)^c}\frac{v(t,y)^{p-1}}{|x_0-y|^{d+sp}}\,dy\,dt\right)^p.
\end{align*}

Combining with both conclusions above, we obtain the desired assertion.
\end{proof}

\begin{proposition}\label{Cl}
Fix $p>1$. Let $R>0$ and $R/2\le r<r+\rho\le R$ with $r+\rho/2\le3R/4$. Let $u$ be a subsolution to $\partial_tu-\sL_tu=0$ in $I^\ominus_R(t_0)\times B_R(x_0)$ with $t_0\in \R$ and $x_0\in\R^d$.  For $l, C>0$, define \begin{equation}\label{vl}
v(t,x)=v_l(t,x):=(u(t,x)-k(t))_+,\qquad k(t):=C\int^t_{t_0-(r+\rho)^{sp}}\int_{B_R(x_0)^c}\frac{u_+(s,y)^{p-1}}{|x_0-y|^{d+sp}}\,dy\,ds+l.
\end{equation}   Then, for any $\xi>0$, we can choose a proper positive constant $C$ in the definition of \eqref{vl} so that there exists a constant $C_3>0$ $($which depends on $p$ and $\xi$ only$)$ such that
\begin{equation*}\begin{split}
&\sup_{t\in I^\ominus_r(t_0)}\int_{B_{r}(x_0)}v_l(t,x)^{1+\xi}\,dx\\
&+\int_{I^\ominus_{r}(t_0)}\int_{B_{r}(x_0)} \int_{B_{r}(x_0)} |v_l^{(p-1+\xi)/p}(t,x)-v_l^{(p-1+\xi)/p}(t,y)|^p J(t;x,y)\,dx\,dy\,dt\\
&\le C_3 \rho^{-sp}\int_{I^\ominus_{r+\rho}(t_0)}\int_{B_{r+\rho}(x_0)} v_l(t,x)^{p-1+\xi}\,dx\,dt\\
&\quad +C_3 \left[(r+\rho)^{sp}-r^{sp}\right]^{-1}\int_{I^\ominus_{r+\rho}(t_0)}\int_{B_{r+\rho}(x_0)} v_l(t,x)^{1+\xi}\,dx\,dt\\
&\quad+C_3\rho^{-sp}\left(\frac{R}{\rho}\right)^d\left(\int_{I^\ominus_{r+\rho}(t_0)}\int_{B_{r+\rho}(x_0)} v_l(t,x)^{\xi}\,dx\,dt\right)\left(\sup_{t\in I^\ominus_R(t_0)}\oint_{B_R(x_0)}u_+(t,y)^{p-1}\,dy\right).
\end{split}\end{equation*}
\end{proposition}

\begin{proof}
Let $R/2\le r<r+\rho\le R$ with $r+\rho/2\le3R/4$. Similar to the proof of
Proposition \ref{sup}, introduce the function $\phi(t,x)=\phi_1(x)\phi_2(t)$ defined on $B_{r+\rho}(x_0)\times(t_0-(r+\rho)^{sp},t_0)$, where $\phi_1: B_{r+\rho}(x_0)\to [0,1]$ is a smooth cut-off function of $B_r(x_0)\subset B_{r+\rho/2}(x_0)$ and $\phi_2: (t_0-(r+\rho)^{sp},t_0)\to [0,1]$ is differentiable such that $\phi_2(t)=1$ for all $t\in (t_0-r^{sp},t_0)$ and $\phi_2(t_0-(r+\rho)^{sp})=0$. In particular,
$$|D\phi|\le c_0\rho^{-1}
\quad \text{and}
\quad|\partial_t\phi|\le c_0\left[(r+\rho)^{sp}-r^{sp}\right]^{-1}.$$
It is clear that $k'(t)\ge0$ for all $t\in (t_0-(r+\rho)^{sp},t_0).$
Apply the Caccioppoli-type inequality established in Lemma \ref{2} with $\tau_2=t_0$, $\tau_1=t_0-r^{sp}$, $h=(r+\rho)^{sp}-r^{sp}$ and the function $\phi(t,x)$ as well as the function $v(t,x)$ defined by \eqref{vl}, where $C>0$ is a constant that will be chosen suitably. As a result, we have
\begin{equation}\label{a}\begin{split}
&\sup_{t\in I^\ominus_r(t_0)}\int_{B_{r}(x_0)}v_l(t,x)^{1+\xi}\,dx\\
&\quad+\int_{I^\ominus_{r}(t_0)}\int_{B_{r}(x_0)} \int_{B_{r}(x_0)}|v_l^{(p-1+\xi)/p}(t,x)-v_l^{(p-1+\xi)/p}(t,y)|^p J(t;x,y)\,dx\,dy\,dt\\
&\le c_1\sup_{t\in I^\ominus_ 
{r+\rho}(t_0)}\int_{B_{r+\rho}(x_0)}v_l(t,x)^{1+\xi}\phi(t,x)^p\,dx\\
&\quad+c_1\int_{I^\ominus_{r+\rho}(t_0)}\int_{B_{r+\rho}(x_0)} \int_{B_{r+\rho}(x_0)} |v_l^{(p-1+\xi)/p}(t,x)-v_l^{(p-1+\xi)/p}(t,y)|^p \phi(t,x)^p J(t;x,y)\,dx\,dy\,dt\\
&\le c_2\int_{I^\ominus_{r+\rho}(t_0)}\left(\int_{B_{r+\rho}(x_0)} v_l(t,x)^{\xi}\phi(t,x)^p\,dx\right)\left(\sup_{x\in B_{r+\rho/2}(x_0)}\int_{\R^d\setminus {B_{r+\rho}(x_0)}}v_l(t,y)^{p-1}J(t;x,y)\,dy\right)\,dt\\
&\quad+c_2\int_{I^\ominus_{r+\rho}(t_0)}\int_{B_{r+\rho}(x_0)}\int_{B_{r+\rho}(x_0)} \max\{v_l^{(p-1+\xi)/p}(t,x),v_l^{(p-1+\xi)/p}(t,y )\}^p|\phi(t,x)-\phi(t,y)|^pJ(t;x,y)\,dx\,dy\,dt\\
&\quad+c_2\int_{I^\ominus_{r+\rho}(t_0)}\int_{B_{r+\rho}
(x_0)}|\partial_t\phi(t,x)^{p}|v_l(t,x)^{1+\xi}\,dx\,dt\\
&\quad-C\int_{I^\ominus_{r+\rho}(t_0)}\left(\int_{B_{r+\rho}(x_0)}v_l(t,x)^{\xi}\phi(t,x)^
p\,dx\right)\left(\int_{B_R(x_0)^c}\frac{u_+(t,y)^{p-1}}{|x_0-y|^{d+sp}}\,dy\right)\,dt,
\end{split}\end{equation}
where the constant $c_2$ is independent of $C$.

The four terms on the right-hand side of \eqref{a} are treated as follows. The second term follows from the argument for the estimate of $H_2$
in the proof of Proposition \ref{sup}:
\begin{align*}
&\int_{I^\ominus_{r+\rho}(t_0)}\int_{B_{r+\rho}(x_0)}\int_{B_{r+\rho}(x_0)} \max\{
v_l^{(p-1+\xi)/p}(t,x),
v_l^{(p-1+\xi)/p}(t,y )\}^p|\phi(t,x)-\phi(t,y)|^pJ(t;x,y)\,dx\,dy\,dt\\
&\le c_3 \rho^{-sp}\int_{I^\ominus_{r+\rho}(t_0)}\int_{B_{r+\rho}(x_0)} v_l(t,x)^{p-1+\xi}\,dx\,dt.
\end{align*}
For the third term,
$$\int_{I^\ominus_{r+\rho}(t_0)}\int_{B_{r+\rho}
(x_0)}|\partial_t\phi(t,x)^{p}|
v_l(t,x)^{1+\xi}\,dx\,dt\le c_4\left[(r+\rho)^{sp}-r^{sp}\right]^{-1}\int_{I^\ominus_{r+\rho}(t_0)}\int_{B_{r+\rho}(x_0)} v_l(t,x)^{1+\xi}\,dx\,dt.$$
The first term and the last negative term need to be handled carefully. To this end, we decompose the tail term into an intermediate tail and a remaining tail:
\begin{align*}
&\sup_{x\in B_{r+\rho/2}(x_0)}\int_{\R^d\setminus {B_{r+\rho}(x_0)}}v_l(t,y)^{p-1}J(t;x,y)\,dy\\
&\le\sup_{x\in B_{r+\rho/2}(x_0)}\int_{\R^d\setminus {B_{r+\rho}(x_0)}}
u_+(t,y)^{p-1}J(t;x,y)\,dy\\
&=\sup_{x\in B_{r+\rho/2}(x_0)}\int_{B_R(x_0)\setminus {B_{r+\rho}(x_0)}}
u_+(t,y)^{p-1}J(t;x,y)\,dy+\sup_{x\in B_{r+\rho/2}(x_0)}\int_{\R^d\setminus {B_R(x_0)}}u_+(t,y)^{p-1}J(t;x,y)\,dy\\
&\le c_5\rho^{-d-sp}\int_{B_R(x_0)}
u_+(t,y)^{p-1}\,dy+c_5\int_{\R^d\setminus {B_R(x_0)}}\frac{
u_+(t,y)^{p-1}}{|x_0-y|^{d+sp}}\,dy\\
&\le c_6\rho^{-sp}\left(\frac{R}{\rho}\right)^d\oint_{B_R(x_0)}
u_+(t,y)^{p-1}\,dy+c_6\int_{\R^d\setminus {B_R(x_0)}}\frac{
u_+(t,y)^{p-1}}{|x_0-y|^{d+sp}}\,dy,
\end{align*}
where in the first inequality we used the fact that $v_l(t,y)\le u_+(t,y)$, and in the second inequality we used the fact that for all $x\in  B_{r+\rho/2}(x_0)$ and
$y\in\R^d\setminus B_R(x_0)$,
$$\frac{|x_0-y|}{|x-y|}\le 1+\frac{|x_0-x|}{|x-y|}\le 1+\frac{r+\rho/2}{R-r-\rho/2}\le 1+\frac{3R/4}{R/4}=4.$$
Then, we
obtain
\begin{align*}
&c_2\int_{I^\ominus_{r+\rho}(t_0)}\left(\int_{B_{r+\rho}(x_0)} v_l(t,x)^{\xi}\phi(t,x)^p\,dx\right)\left(\sup_{x\in B_{r+\rho/2}(x_0)}\int_{\R^d\setminus {B_{r+\rho}(x_0)}}v_l(t,y)^{p-1}J(t;x,y)\,dy\right)\,dt\\
&\le c_2c_6\rho^{-sp}\left(\frac{R}{\rho}\right)^d\left(\int_{I^\ominus_{r+\rho}(t_0)}\int_{B_{r+\rho}(x_0)}v_l(t,x)^{\xi}\phi(t,x)^p\,dx\,dt\right)\left(\sup_{t\in I^\ominus_R(t_0)}\oint_{B_R(x_0)}u_+(t,y)^{p-1}\,dy\right)\\
&\quad+c_2c_6\int_{I^\ominus_{r+\rho}(t_0)}\left(\int_{B_{r+\rho}(x_0)} v_l(t,x)^{\xi}\phi(t,x)^p\,dx\right)\left(\int_{B_R(x_0)^c}\frac{
u_+(t,y)^{p-1}}{|x_0-y|^{d+sp}}\,dy\right)\,dt.
\end{align*}
The last term in the above estimate will be cancelled with the last negative term on the right-hand side of \eqref{a}, if we choose $C:=c_2c_6$.

Combining with all these estimates for the right-hand side of \eqref{a}, we obtain
\begin{equation*}\begin{split}
&\sup_{t\in I^\ominus_r(t_0)}\int_{B_{r}(x_0)}v_l(t,x)^{1+\xi}\,dx\\
&+\int_{I^\ominus_{r}(t_0)}\int_{B_{r}(x_0)} \int_{B_{r}(x_0)} |v_l^{(p-1+\xi)/p}(t,x)-v_l^{(p-1+\xi)/p}(t,y)|^p J(t;x,y)\,dx\,dy\,dt\\
&\le c_7 \rho^{-sp}\int_{I^\ominus_{r+\rho}(t_0)}\int_{B_{r+\rho}(x_0)} v_l(t,x)^{p-1+\xi}\,dx\,dt\\
&\quad +c_7 \left[(r+\rho)^{sp}-r^{sp}\right]^{-1}\int_{I^\ominus_{r+\rho}(t_0)}\int_{B_{r+\rho}(x_0)} v_l(t,x)^{1+\xi}\,dx\,dt\\
&\quad+c_7\rho^{-sp}\left(\frac{R}{\rho}\right)^d\left(\int_{I^\ominus_{r+\rho}(t_0)}\int_{B_{r+\rho}(x_0)} v_l(t,x)^{\xi}\,dx\,dt\right)\left(\sup_{t\in I^\ominus_R(t_0)}\oint_{B_R(x_0)}u_+(t,y)^{p-1}\,dy\right),
\end{split}\end{equation*}
where the last term on the right-hand side of the above inequality is based on the fact that $0\le \phi(t,x)\le1$. The proof is finished.
\end{proof}

\section{Proof of Theorem \ref{main1}}\label{section3}
For $l, C>0$, recalling $v_l(t,x)$ as defined by \eqref{vl}, we can rewrite $$v_l(t,x)=\left(\tilde u(t,x)-l\right)_+,$$
where
$$\tilde u(t,x):=u(t,x)-C\int^t_{t_0-(r+\rho)^{sp}}\int_{B_R(x_0)^c}\frac{u_+(s,y)^{p-1}}{|x_0-y|^{d+sp}}\,dy\,ds.$$ In the following, we always take $C>0$ to be that in Proposition \ref{Cl}.
Obviously, for any $l_2\ge l_1>0$, we have $v_{l_1}(t,x)\ge v_{l_2}(t,x)$.
Furthermore, if $v_{l_2}(t,x)>0$ then $\tilde u(t,x)>l_2$ and so $v_{l_1}(t,x)=\tilde u(t,x)-l_1>l_2-l_1$. Hence, for any $l_2\ge l_1>0$ and $m_2\ge m_1>0$, it follows that
\begin{equation}\label{L}
v_{l_2}(t,x)^{m_1}\le v_{l_1}(t,x)^{m_2}/(l_2-l_1)^{m_2-m_1}.
\end{equation}
Note that
$\tilde u(t,x)-l$ is
a subsolution to $\partial_tu-\sL_tu=0$ in $I^\ominus_R(t_0)\times B_R(x_0)$, and so $v_l(t,x)$ is also a subsolution to $\partial_tu-\sL_tu=0$ in $I^\ominus_R(t_0)\times B_R(x_0)$, thanks to Lemma \ref{Z}.

\medskip

Now, it is a position to present the

\begin{proof}[Proof of Theorem $\ref{main1}$] (1) We first consider $p\ge 2$, and the proof is split into three steps

(i) Recall that $R/2\le r<r+\rho\le R$ with $r+\rho/2\le3R/4$.
In particular, $\rho/2\le 3R/4-R/2=R/4$, and so $\rho\le R/2\le r$.
According to Lemma \ref{L:lemma3.3}\,(i) with $f(t,x)=v_l(t,x)$,
\begin{equation*}\begin{split}
&\int_{I^\ominus_r(t_0)} \oint_{B_r(x_0)}v_l(t,x)^{(1+2s/d)p}\,dx\,dt\\
&\le c_1\bigg[r^{sp}\int_{I^\ominus_r(t_0)}\int_{B_r(x_0)} \oint_{B_r(x_0)}|v_l(t,x)-v_l(t,y)|^p J(t;x,y)\,dx\,dy\,dt\\
&\qquad\quad+\int_{I^\ominus_r(t_0)} \oint_{B_r(x_0)}
v_l(t,x)^{p}\,dx\,dt\bigg]\times \left(\sup_{I^\ominus_r(t_0)}\oint_{B_r(x_0)}
v_l(t,x)^2\,dx\right)^{sp/d}\\
&=:c_1\left(J_1+J_2\right)\times J_3^{sp/d},
\end{split}\end{equation*}
where
$$J_1=r^{sp}\int_{I^\ominus_r(t_0)}\int_{B_r(x_0)} \oint_{B_r(x_0)} |v_l(t,x)-v_l(t,y)|^p J(t;x,y)\,dx\,dy\,dt
$$
and
$$J_2=\int_{I^\ominus_r(t_0)} \oint_{B_r(x_0)} v_l(t,x)^{p}\,dx\,dt,\qquad J_3=\sup_{I^\ominus_r(t_0)}\oint_{B_r(x_0)}
v_l(t,x)^2\,dx.$$

According to Proposition \ref{Cl} with $\xi=1$,
\begin{align*}
J_1\le&
c_2r^{sp}\bigg[\rho^{-sp}\int_{I^\ominus_{r+\rho}(t_0)} \oint_{B_{r+\rho}(x_0)}v_l(t,x)^p\,dx\,dt+\left[(r+\rho)^{sp}-r^{sp}\right]^{-1}\int_{I^\ominus_{r+\rho}(t_0)} \oint_{B_{r+\rho}(x_0)}v_l(t,x)^2\,dx\,dt
\\
&\qquad\quad+\rho^{-sp}\left(\frac{R}{\rho}\right)^d\left(\int_{I^\ominus_{r+\rho}(t_0)} \oint_{B_{r+\rho}(x_0)}v_l(t,x) \,dx\,dt\right)\bigg(\sup_{t\in I^\ominus_R(t_0)}\oint_{B_R(x_0)}
u_+(t,y)^{p-1}\,dy\bigg)\bigg]. \end{align*} It is obvious that
\begin{align*}
J_2\le &r^{sp}\rho^{-sp}\int_{I^\ominus_{r+\rho}(t_0)} \oint_{B_{r+\rho}(x_0)}v_l(t,x)^p\,dx\,dt\\
\le&
r^{sp}\bigg[\rho^{-sp}\int_{I^\ominus_{r+\rho}(t_0)} \oint_{B_{r+\rho}(x_0)}v_l(t,x)^p\,dx\,dt+\left[(r+\rho)^{sp}-r^{sp}\right]^{-1}\int_{I^\ominus_{r+\rho}(t_0)} \oint_{B_{r+\rho}(x_0)}v_l(t,x)^2\,dx\,dt
\\
&\quad\quad+\rho^{-sp}\left(\frac{R}{\rho}\right)^d\left(\int_{I^\ominus_{r+\rho}(t_0)} \oint_{B_{r+\rho}(x_0)}v_l(t,x) \,dx\,dt\right)\bigg(\sup_{t\in I^\ominus_R(t_0)}\oint_{B_R(x_0)}
u_+(t,y)^{p-1}\,dy\bigg)\bigg].
\end{align*}
On the other hand, by Proposition \ref{Cl} with $\xi=1$ again, we get
\begin{align*}
&J_3\le  c_{3}\bigg[\rho^{-sp}\int_{I^\ominus_{r+\rho}(t_0)} \oint_{B_{r+\rho}(x_0)}v_l(t,x)^p\,dx\,dt+\left[(r+\rho)^{sp}-r^{sp}\right]^{-1}\int_{I^\ominus_{r+\rho}(t_0)} \oint_{B_{r+\rho}(x_0)}v_l(t,x)^2\,dx\,dt
\\
&\qquad\qquad\,\,+\rho^{-sp}\left(\frac{R}{\rho}\right)^d\left(\int_{I^\ominus_{r+\rho}(t_0)} \oint_{B_{r+\rho}(x_0)}v_l(t,x) \,dx\,dt\right)\bigg(\sup_{t\in I^\ominus_R(t_0)}\oint_{B_R(x_0)}
u_+(t,y)^{p-1}\,dy\bigg)\bigg].
\end{align*}
Putting all the inequalities together, we find that
\begin{align*}
&\oint_{I^\ominus_r(t_0)} \oint_{B_r(x_0)}v_l(t,x)^{(1+2s/d)p}\,dx\,dt
\le c_1
r^{-sp}\left(J_1+J_2\right)\times J_3^{sp/d}\\
&\le c_{4}\bigg[\rho^{-sp}\int_{I^\ominus_{r+\rho}(t_0)} \oint_{B_{r+\rho}(x_0)}v_l(t,x)^p\,dx\,dt+\left[(r+\rho)^{sp}-r^{sp}\right]^{-1}\int_{I^\ominus_{r+\rho}(t_0)} \oint_{B_{r+\rho}(x_0)}v_l(t,x)^2\,dx\,dt
\\
&\qquad\quad\,\,+\rho^{-sp}\left(\frac{R}{\rho}\right)^d\left(\int_{I^\ominus_{r+\rho}(t_0)} \oint_{B_{r+\rho}(x_0)}v_l(t,x) \,dx\,dt\right)\bigg(\sup_{t\in I^\ominus_R(t_0)}\oint_{B_R(x_0)}
u_+(t,y)^{p-1}\,dy\bigg)\bigg]^{{1+sp/d}}\\
&\le c_5\bigg[\rho^{-sp}R^{sp}\oint_{I^\ominus_{r+\rho}(t_0)} \oint_{B_{r+\rho}(x_0)}v_l(t,x)^p\,dx\,dt+\left[(r+\rho)^{sp}-r^{sp}\right]^{-1}R^{sp}\oint_{I^\ominus_{r+\rho}(t_0)} \oint_{B_{r+\rho}(x_0)}v_l(t,x)^2\,dx\,dt
\\
&\qquad\quad\,\, +\left(\frac{R}{\rho}\right)^{d+sp}\left(\oint_{I^\ominus_{r+\rho}(t_0)} \oint_{B_{r+\rho}(x_0)}v_l(t,x) \,dx\,dt\right)\bigg(\sup_{t\in I^\ominus_R(t_0)}\oint_{B_R(x_0)}
u_+(t,y)^{p-1}\,dy\bigg)\bigg]^{1+sp/d},
\end{align*}
where in the last inequality we used the fact that $R/2\le r+\rho\le R$.

(ii) Now let $0<k<l$ be arbitrary. By the definition of $v_{l}(t,x)$ and \eqref{L},
it holds that
\begin{equation*}\begin{split}
&\oint_{I^\ominus_{r+\rho}(t_0)} \oint_{B_{r+\rho}(x_0)} v_{\frac{l+k}{2}}(t,x)^p\,dx\,dt\le\oint_{I^\ominus_{r+\rho}(t_0)} \oint_{B_{r+\rho}(x_0)}v_k(t,x)^p\,dx\,dt,\\
&\oint_{I^\ominus_{r+\rho}(t_0)} \oint_{B_{r+\rho}(x_0)}  v_{\frac{l+k}{2}}(t,x) \,dx\,dt\le\left(\frac{l-k}{2}\right)^{-(p-1)}\oint_{I^\ominus_{r+\rho}(t_0)} \oint_{B_{r+\rho}(x_0)}v_k(t,x)^p\,dx\,dt,\\
&\oint_{I^\ominus_{r+\rho}(t_0)} \oint_{B_{r+\rho}(x_0)} v_{\frac{l+k}{2}}(t,x)^{2}\,dx\,dt\le\left(\frac{l-k}{2}\right)^{-(p-2)}\oint_{I^\ominus_{r+\rho}(t_0)} \oint_{B_{r+\rho}(x_0)}v_k(t,x)^p\,dx\,dt,
\end{split}\end{equation*} where in the last inequality we used the fact that $p\ge2$.
Then, by applying the last inequality in Step (i) with $(l+k)/2$ in place of $l$ and using
the three estimates above, we find that
\begin{align*}
&\oint_{I^\ominus_r(t_0)} \oint_{B_r(x_0)} v_{\frac{l+k}{2}}(t,x)^{(1+2s/d)p}\,dx\,dt\\
&\le c_5\bigg[\rho^{-sp}R^{sp}\oint_{I^\ominus_{r+\rho}(t_0)} \oint_{B_{r+\rho}(x_0)}v_{\frac{l+k}{2}}(t,x)^p\,dx\,dt+\left[(r+\rho)^{sp}-r^{sp}\right]^{-1}R^{sp}\oint_{I^\ominus_{r+\rho}(t_0)} \oint_{B_{r+\rho}(x_0)}v_{\frac{l+k}{2}}(t,x)^2\,dx\,dt
\\
&\qquad\quad\,\, +\left(\frac{R}{\rho}\right)^{d+sp}\left(\oint_{I^\ominus_{r+\rho}(t_0)} \oint_{B_{r+\rho}(x_0)}  v_{\frac{l+k}{2}}(t,x)\,dx\,dt\right)\bigg(\sup_{t\in I^\ominus_R(t_0)}\oint_{B_R(x_0)}
u_+(t,y)^{p-1}\,dy\bigg)\bigg]^{{1+sp/d}}\\
&\le c_5\bigg[\rho^{-sp}R^{sp}\oint_{I^\ominus_{r+\rho}(t_0)} \oint_{B_{r+\rho}(x_0)}v_{k}(t,x)^p\,dx\,dt\\
&\qquad\quad\,\,+\left[(r+\rho)^{sp}-r^{sp}\right]^{-1}R^{sp}\left(\frac{l-k}{2}\right)^{-(p-2)}\oint_{I^\ominus_{r+\rho}(t_0)} \oint_{B_{r+\rho}(x_0)}v_{k}(t,x)^p\,dx\,dt\\
&\qquad\quad\,\, +\left(\frac{R}{\rho}\right)^{d+sp}\left(\frac{l-k}{2}\right)^{-(p-1)}\left(\oint_{I^\ominus_{r+\rho}(t_0)} \oint_{B_{r+\rho}(x_0)}v_{k}(t,x)^p\,dx\,dt\right)\bigg(\sup_{t\in I^\ominus_R(t_0)}\oint_{B_R(x_0)}
u_+(t,y)^{p-1}\,dy\bigg)\bigg]^{{1+sp/d}}.
\end{align*}
On the other hand, by \eqref{L} again, it holds that
\begin{equation*}
\oint_{I^\ominus_r(t_0)} \oint_{B_r(x_0)} v_{\frac{l+k}{2}}(t,x)^{(1+2s/d)p}\,dx\,dt
\ge\left(\frac{l-k}{2}\right)^{2sp/d}\oint_{I^\ominus_r(t_0)} \oint_{B_r(x_0)}v_l(t,x)^p\,dx\,dt.
\end{equation*}
Thus,
\begin{align*}
&\left(\frac{l-k}{2}\right)^{2sp/d}\oint_{I^\ominus_r(t_0)} \oint_{B_r(x_0)}v_l(t,x)^p\,dx\,dt\\
&\le c_5\bigg[\rho^{-sp}R^{sp}+\left[(r+\rho)^{sp}-r^{sp}\right]^{-1}R^{sp}\left(\frac{l-k}{2}\right)^{-(p-2)}\\
&\quad\qquad
+\left(\frac{R}{\rho}\right)^{d+sp}\left(\frac{l-k}{2}\right)^{-(p-1)}\sup_{t\in I^\ominus_R(t_0)}\oint_{B_R(x_0)}
u_+(t,y)^{p-1}\,dy\bigg]^{{1+sp/d}}\\
&\quad\,\times\left(\oint_{I^\ominus_{r+\rho}(t_0)} \oint_{B_{r+\rho}(x_0)}v_{k}(t,x)^p\,dx\,dt\right)^{{1+sp/d}}.
\end{align*}

(iii) The plan for the remainder of the proof is to iterate the above estimate. Let us now set up the iteration scheme. Define two sequences $r_i=(1+2^{-i})R/2$ and $l_i=M(1-2^{-i})$ for all $i\ge0$, where $M>0$ is to be determined later. It is clear that  $r_i\downarrow R/2$ and $l_i\uparrow M$ as $i\rightarrow\infty$. In particular, $r_0=R$, $l_0=0$ and $$r_{i+1}=r_i-\rho_{i+1}=R\left(1+2^{-i-1}\right)/2,\quad i\ge0,$$ where $\rho_i=2^{-i-1}R$; moreover, by the mean value theorem,  $\left[(r_i+\rho_i)^{sp}-r_i^{sp}\right]^{-1}\le c_6R^{1-sp}\rho_i^{-1}$ for $i\ge0$.
Define $$A_i=\oint_{I^\ominus_{r_i}(t_0)} \oint_{B_{r_i}(x_0)} v_{l_i}(t,x)^p\,dx\,dt,\quad i\ge 0.$$
It holds that for all $i\ge1$,
\begin{align*}
A_i&\le c_7\left(\frac{l_i-l_{i-1}}{2}\right)^{-2sp/d} \bigg[\rho_i^{-sp}R^{sp}
+\rho_i^{-1}R\left(\frac{l_i-l_{i-1}}{2}\right)^{-(p-2)}\\
&\qquad\qquad\qquad\qquad\qquad\quad
+\left(\frac{R}{\rho_i}\right)^{d+sp}\left(\frac{l_i-l_{i-1}}{2}\right)^{-(p-1)}\sup_{t\in I^\ominus_R(t_0)}\oint_{B_R(x_0)}u_+(t,x)^{p-1}\,dx\bigg]^{1+sp/d}A_{i-1}^{1+sp/d}\\
&\le c_8\left(\frac{2^{i+1}}{M}\right)^{2sp/d} \bigg[2^{(i+1)sp}
+ \frac{2^{(i+1)(p-1)}}{M^{p-2}}
+\frac{2^{(i+1)(p-1)}2^{(i+1)(d+sp)}}{M^{p-1}}\sup_{t\in I^\ominus_R(t_0)}\oint_{B_R(x_0)}u_+(t,x)^{p-1}\,dx\bigg]^{1+sp/d}A_{i-1}^{1+sp/d}\\
&\le c_9\frac{2^{\gamma i}}{M^{2sp/d}} \bigg[1+\frac{1}{M^{p-2}}
+\frac{\sup_{t\in I^\ominus_R(t_0)}\oint_{B_R(x_0)}u_+(t,x)^{p-1}\,dx}{M^{p-1}}\bigg]^{1+sp/d}A_{i-1}^{1+sp/d},
 \end{align*}
where $\gamma=2sp/d+(d+sp+p-1)(1+sp/d)$.

In the following, we choose $$M=\sup_{t\in I^\ominus_R(t_0)}\oint_{B_R(x_0)}u_+(t,x)^{p-1}\,dx+\left\{C_0\left[1+2\left(\sup_{t\in I^\ominus_R(t_0)}\oint_{B_R(x_0)}u_+(t,x)^{p-1}\,dx\right)^{2-p}\right]^{1+d/(sp)}A_0 \right\}^{1/2},$$
 where $$C_0=2^{\gamma(d/(sp))^2}(c_92^\gamma)^{d/(sp)}.$$
 Then, we have
\begin{align*}A_0\le &C_0^{-1}M^2\left[1+2\left(\sup_{t\in I^\ominus_R(t_0)}\oint_{B_R(x_0)}u_+(t,x)^{p-1}\,dx\right)^{2-p}\right]^{-1-d/(sp)}\\
\le &C_0^{-1}M^2\left[1+2M^{2-p}\right]^{-1-d/(sp)}\le C_0^{-1}M^2\left(1+M^{2-p}+\left(\sup_{t\in I^\ominus_R(t_0)}\oint_{B_R(x_0)}u_+(t,x)^{p-1}\,dx\right)M^{1-p}\right)^{-1-d/(sp)}\\
=&\left(2^{\gamma}\right)^{-(d/(sp))^2} \left(\frac{c_92^\gamma}{M^{2sp/d}}\left[1+\frac{1}{M^{p-2}}+\frac{\sup_{t\in I^\ominus_R(t_0)}\oint_{B_R(x_0)}u_+(t,x)^{p-1}\,dx}{M^{p-1}}\right]^{1+sp/d}\right)^{-d/(sp)},\end{align*} where in the
last two
inequalities we used the facts that $p\ge2$ and $M\ge \sup_{t\in I^\ominus_R(t_0)}\oint_{B_R(x_0)}u_+(t,x)^{p-1}\,dx$.
This, along with Lemma \ref{iteration} with $Y_j=A_j$
and $\beta=sp/d$, yields that
$$A_i\rightarrow0,\quad i\rightarrow\infty.$$
Hence, we have
\begin{align*}
&\sup_{B_{R/2}(x_0)\times I^\ominus_{R/2}(t_0)}\left(u-C\int_{I^\ominus_{R/2}(t_0)}\int_{B_R(x_0)^c}\frac{u_+(t,y)^{p-1}}{|x_0-y|^{d+sp}}dydt\right)\\
&\le
\sup_{t\in I^\ominus_R(t_0)}\oint_{B_R(x_0)}u_+(t,x)^{p-1}\,dx+\left\{C_0\left[1+2\left(\sup_{t\in I^\ominus_R(t_0)}\oint_{B_R(x_0)}u_+(t,x)^{p-1}\,dx\right)^{2-p}\right]^{1+d/(sp)}A_0 \right\}^{1/2}.
 \end{align*}
In particular,
noting that $$\left(u(t,x)-C\int^t_{t_0-R^{sp}}\int_{B_R(x_0)^c}\frac{u_+(s,y)^{p-1}}{|x_0-y|^{d+sp}}\,dy\,ds\right)_+^p
\le u_+(t,x)^p,$$
we have
\begin{equation}\label{J}\begin{split}
&\sup_{B_{R/2}(x_0)\times I^\ominus_{R/2}(t_0)}u\\
&\le  \sup_{t\in I^\ominus_R(t_0)}\oint_{B_R(x_0)}u_+(t,x)^{p-1}\,dx  +C{\rm Tail}_{p-1}^{p-1}(u; x_0,R, t_0-R^{sp}, t_0)
\\
&\quad +c_{10}\left[1+2\left(\sup_{t\in I^\ominus_R(t_0)}\oint_{B_R(x_0)}u_+(t,x)^{p-1}\,dx\right)^{2-p}\right]^
{(1+d/(sp))/2}\left(\oint_{I^\ominus_R(t_0)}\oint_{ B_R(x_0)}u_+(t,x)^p\,dx\,dt\right)^{1/2}\\
&\le  \sup_{t\in I^\ominus_R(t_0)}\oint_{B_R(x_0)}u_+(t,x)^{p-1}\,dx  +C{\rm Tail}_{p-1}^{p-1}(u; x_0,R, t_0-R^{sp}, t_0)
\\
&\quad +c_{10}\left[1+2\left(\oint_{I^\ominus_R(t_0)}\oint_{B_R(x_0)}u_+(t,x)^{p-1}\,dx\,dt\right)^{2-p}\right]^
{(1+d/(sp))/2}\left(\oint_{I^\ominus_R(t_0)}\oint_{ B_R(x_0)}u_+(t,x)^p\,dx\,dt\right)^{1/2},
\end{split}\end{equation} where in the last inequality we used the facts that $p\ge2$ and
$$\oint_{I^\ominus_R(t_0)}\oint_{B_R(x_0)}u_+(t,x)^{p-1}\,dx\,dt\le \sup_{t\in I^\ominus_R(t_0)}\oint_{B_R(x_0)}u_+(t,x)^{p-1}\,dx.$$

Since $u$ is a subsolution to $\partial_tu-\sL_tu=0$ in $I^\ominus_R(t_0)\times B_R(x_0)$, $u_+$ is also a subsolution to $\partial_tu-\sL_tu=0$ in $I^\ominus_R(t_0)\times B_R(x_0)$ thanks to
Lemma \ref{Z}.
Hence, for the first term in the right hand side of the inequality above, by applying the H\"{o}lder inequality and Proposition \ref{sup} to $u_+$, we get
\begin{align}\label{K}
 \sup_{t\in I^\ominus_R(t_0)}\oint_{B_R(x_0)}u_+(t,x)^{p-1}\,dx
&\le\sup_{I^\ominus_R(t_0)}\left(\oint_{B_R(x_0)}u_+(t,x)^p\,dx\right)^{(p-1)/{p}}\nonumber\\
&\le c_{11}\left(\oint_{I^\ominus_R(t_0)}\oint_{B_R(x_0)}u_+(t,x)^{2p-2}\,dx\,dt+\oint_{I^\ominus_R(t_0)}\oint_{B_R(x_0)}u_+(t,x)^p\,dx\,dt\right)^{(p-1)/{p}}\\
&\quad\,\,\,+
c_{11}\left({\rm Tail}_{p-1}^{p-1}(u; x_0,R, t_0-R^{sp}, t_0)\right)^{p-1}\nonumber.
\end{align}
Combining \eqref{J} with \eqref{K},
we obtain
\begin{align*}
&\sup_{B_{R/2}(x_0)\times I^\ominus_{R/2}(t_0)}u\\
&\le c_{12}\left(\oint_{I^\ominus_{R}(t_0)}\oint_{B_{R}(x_0)}u_+(t,x)^{2p-2}\,dx\,dt+\oint_{I^\ominus_{R}(t_0)}\oint_{B_{R}(x_0)}u_+(t,x)^p\,dx\,dt\right)^{{(p-1)}/{p}}\\
&\quad\,\,\,+c_{12}\left[1+2\left(\oint_{I^\ominus_R(t_0)}\oint_{B_R(x_0)}u_+(t,x)^{p-1}\,dx\,dt\right)^{2-p}\right]^
{(1+d/(sp))/2}\left(\oint_{I^\ominus_R(t_0)}\oint_{ B_R(x_0)}u_+(t,x)^p\,dx\,dt\right)^{1/2}\\
&\quad\,\,\,+
c_{12}{\rm Tail}_{p-1}^{p-1}(u; x_0,R, t_0-{R}^{sp}, t_0)+c_{12}\left({\rm Tail}_{p-1}^{p-1}(u; x_0,R, t_0-R^{sp}, t_0)\right)^{p-1}.
\end{align*}
The proof for the case that $p\ge2$ is complete.

(2) Next we consider $p\in (1,2)$, and the proof is also split into three steps. For the safe of completeness and for the comparison with the case that $p\in [2,\infty)$, we present some details here.

(i) By applying Lemma \ref{L:lemma3.3}\,(ii) to $v_l^{(p-1+\xi)/p}(t,x)$ with $\xi>0$, we find that
\begin{equation*}\begin{split}
&\int_{I^\ominus_r(t_0)} \oint_{B_r(x_0)}  v_l(t,x)^{(p-1+\xi)(1+sp/d)}\,dx\,dt\\
&\le c_1\bigg[r^{sp}\int_{I^\ominus_r(t_0)}\int_{B_r(x_0)} \oint_{B_r(x_0)}|v_l^{(p-1+\xi)/p}(t,x)-v_l^{(p-1+\xi)/p}(t,y)|^p J(t;x,y)\,dx\,dy\,dt\\
&\qquad\quad+\int_{I^\ominus_r(t_0)} \oint_{B_r(x_0)}
v_l(t,x)^{p-1+\xi}\,dx\,dt\bigg]\times \left(\sup_{I^\ominus_r(t_0)}\oint_{B_r(x_0)}
v_l(t,x)^{p-1+\xi}\,dx\right)^{sp/d}\\
&=:c_1\left(K_1+K_2\right)\times K_3^{sp/d},
\end{split}\end{equation*}
where
$$K_1=r^{sp}\int_{I^\ominus_r(t_0)}\int_{B_r(x_0)} \oint_{B_r(x_0)} |v_l^{(p-1+\xi)/p}(t,x)-v_l^{(p-1+\xi)/p}(t,y)|^p J(t;x,y)\,dx\,dy\,dt
$$
and
$$K_2=\int_{I^\ominus_r(t_0)} \oint_{B_r(x_0)} v_l(t,x)^{p-1+\xi}\,dx\,dt,\qquad K_3=\sup_{I^\ominus_r(t_0)}\oint_{B_r(x_0)}
v_l(t,x)^{p-1+\xi}\,dx.$$

Below, we take $\xi>d(2-p)/(sp)$. Then, $(p-1+\xi)(1+sp/d)>1+\xi$. Thus, according to Proposition \ref{Cl},
\begin{align*}
K_1\le&c_2 r^{sp} \Bigg[ \rho^{-sp}\int_{I^\ominus_{r+\rho}(t_0)}\oint_{B_{r+\rho}(x_0)} v_l(t,x)^{p-1+\xi}\,dx\,dt+ \left[(r+\rho)^{sp}-r^{sp}\right]^{-1}\int_{I^\ominus_{r+\rho}(t_0)}\oint_{B_{r+\rho}(x_0)} v_l(t,x)^{1+\xi}\,dx\,dt\\
&\quad\quad\quad+ \rho^{-sp}\left(\frac{R}{\rho}\right)^d\left(\int_{I^\ominus_{r+\rho}(t_0)}\oint_{B_{r+\rho}(x_0)} v_l(t,x)^{\xi}\,dx\,dt\right)\left(\sup_{t\in I^\ominus_R(t_0)}\oint_{B_R(x_0)}u_+(t,y)^{p-1}\,dy\right)\Bigg].
\end{align*} It holds trivially that
\begin{align*}
K_2\le&r^{sp} \Bigg[\rho^{-sp}\int_{I^\ominus_{r+\rho}(t_0)}\oint_{B_{r+\rho}(x_0)} v_l(t,x)^{p-1+\xi}\,dx\,dt+ \left[(r+\rho)^{sp}-r^{sp}\right]^{-1}\int_{I^\ominus_{r+\rho}(t_0)}\oint_{B_{r+\rho}(x_0)} v_l(t,x)^{1+\xi}\,dx\,dt\\
&\quad\quad\quad+ \rho^{-sp}\left(\frac{R}{\rho}\right)^d\left(\int_{I^\ominus_{r+\rho}(t_0)}\oint_{B_{r+\rho}(x_0)} v_l(t,x)^{\xi}\,dx\,dt\right)\left(\sup_{t\in I^\ominus_R(t_0)}\oint_{B_R(x_0)}u_+(t,y)^{p-1}\,dy\right)\Bigg].\end{align*} On the other hand,
by applying $p\in (1,2)$ and the Jensen inequality, as well as Proposition \ref{Cl} again,
we have
\begin{align*}K_3\le& \left(\sup_{I^\ominus_r(t_0)}\oint_{B_r(x_0)}
v_l(t,x)^{1+\xi}\,dx\right)^{(p-1+\xi)/(1+\xi)}\\
\le&c_3\Bigg[ \rho^{-sp}\int_{I^\ominus_{r+\rho}(t_0)}\oint_{B_{r+\rho}(x_0)} v_l(t,x)^{p-1+\xi}\,dx\,dt+ \left[(r+\rho)^{sp}-r^{sp}\right]^{-1}\int_{I^\ominus_{r+\rho}(t_0)}\oint_{B_{r+\rho}(x_0)} v_l(t,x)^{1+\xi}\,dx\,dt\\
&\quad\quad+ \rho^{-sp}\left(\frac{R}{\rho}\right)^d\left(\int_{I^\ominus_{r+\rho}(t_0)}\oint_{B_{r+\rho}(x_0)} v_l(t,x)^{\xi}\,dx\,dt\right)\left(\sup_{t\in I^\ominus_R(t_0)}\oint_{B_R(x_0)}u_+(t,y)^{p-1}\,dy\right)\Bigg]^{(p-1+\xi)/(1+\xi)}.
\end{align*}
Hence, putting all the estimates together, we find that
\begin{align*}
&\oint_{I^\ominus_r(t_0)} \oint_{B_r(x_0)}  v_l(t,x)^{(p-1+\xi)(1+sp/d)}\,dx\,dt\le c_1r^{-sp}\left(K_1+K_2\right)\times K_3^{sp/d}\\
&\le c_4 \Bigg[ \rho^{-sp}\int_{I^\ominus_{r+\rho}(t_0)}\oint_{B_{r+\rho}(x_0)} v_l(t,x)^{p-1+\xi}\,dx\,dt+ \left[(r+\rho)^{sp}-r^{sp}\right]^{-1}\int_{I^\ominus_{r+\rho}(t_0)}\oint_{B_{r+\rho}(x_0)} v_l(t,x)^{1+\xi}\,dx\,dt\\
&\quad\quad\,\,+ \rho^{-sp}\left(\frac{R}{\rho}\right)^d\left(\int_{I^\ominus_{r+\rho}(t_0)}\oint_{B_{r+\rho}(x_0)} v_l(t,x)^{\xi}\,dx\,dt\right)\left(\sup_{t\in I^\ominus_R(t_0)}\oint_{B_R(x_0)}u_+(t,y)^{p-1}\,dy\right)\Bigg]^{1+ sp(p-1+\xi) /(d(1+\xi))}\\
&\le c_5\Bigg[ \rho^{-sp}R^{sp}\oint_{I^\ominus_{r+\rho}(t_0)}\oint_{B_{r+\rho}(x_0)} v_l(t,x)^{p-1+\xi}\,dx\,dt+ \left[(r+\rho)^{sp}-r^{sp}\right]^{-1}R^{sp}\oint_{I^\ominus_{r+\rho}(t_0)}\oint_{B_{r+\rho}(x_0)} v_l(t,x)^{1+\xi}\,dx\,dt\\
&\quad+ \left(\frac{R}{\rho}\right)^{d+sp}\left(\oint_{I^\ominus_{r+\rho}(t_0)}\oint_{B_{r+\rho}(x_0)} v_l(t,x)^{\xi}\,dx\,dt\right)\left(\sup_{t\in I^\ominus_R(t_0)}\oint_{B_R(x_0)}u_+(t,y)^{p-1}\,dy\right)\Bigg]^{1+ sp(p-1+\xi) / (d(1+\xi))},
\end{align*}
where in the last inequality we used the fact that $R/2\le r+\rho\le R$.

(ii) Now let $0<k<l$ be arbitrary. By the definition of $v_{l}(t,x)$ and \eqref{L},
it holds that
\begin{equation*}\begin{split}
&\oint_{I^\ominus_{r+\rho}(t_0)} \oint_{B_{r+\rho}(x_0)}  v_{\frac{l+k}{2}}(t,x)^{1+\xi}\,dx\,dt\le\oint_{I^\ominus_{r+\rho}(t_0)} \oint_{B_{r+\rho}(x_0)}v_k(t,x)^{1+\xi}\,dx\,dt,\\
&\oint_{I^\ominus_{r+\rho}(t_0)} \oint_{B_{r+\rho}(x_0)}  v_{\frac{l+k}{2}}(t,x)^\xi \,dx\,dt\le\left(\frac{l-k}{2}\right)^{-1}\oint_{I^\ominus_{r+\rho}(t_0)} \oint_{B_{r+\rho}(x_0)}v_k(t,x)^{1+\xi}\,dx\,dt,\\
&\oint_{I^\ominus_{r+\rho}(t_0)} \oint_{B_{r+\rho}(x_0)}  v_{\frac{l+k}{2}}(t,x)^{p-1+\xi}\,dx\,dt\le\left(\frac{l-k}{2}\right)^{-(2-p)}\oint_{I^\ominus_{r+\rho}(t_0)} \oint_{B_{r+\rho}(x_0)}v_k(t,x)^{1+\xi}\,dx\,dt,
\end{split}\end{equation*}
where in the last inequality we used the fact that $p\in(1,2)$.
Then, by applying the last inequality in Step (i) with $(l+k)/2$ in place of $l$ and using
the three estimates above, we find that
\begin{align*}
&\oint_{I^\ominus_r(t_0)} \oint_{B_r(x_0)}  v_{\frac{l+k}{2}}(t,x)^{(p-1+\xi)(1+sp/d)}\,dx\,dt\\
&\le c_5\bigg[\rho^{-sp}R^{sp}\oint_{I^\ominus_{r+\rho}(t_0)} \oint_{B_{r+\rho}(x_0)}v_{\frac{l+k}{2}}(t,x)^{p-1+\xi}\,dx\,dt\\
&\qquad\quad +\left[(r+\rho)^{sp}-r^{sp}\right]^{-1}R^{sp}\oint_{I^\ominus_{r+\rho}(t_0)} \oint_{B_{r+\rho}(x_0)}v_{\frac{l+k}{2}}(t,x)^{1+\xi}\,dx\,dt\\
&\qquad\quad +\left(\frac{R}{\rho}\right)^{d+sp}\left(\oint_{I^\ominus_{r+\rho}(t_0)} \oint_{B_{r+\rho}(x_0)}  v_{\frac{l+k}{2}}(t,x)^\xi\,dx\,dt\right)\bigg(\sup_{t\in I^\ominus_R(t_0)}\oint_{B_R(x_0)}
u_+(t,y)^{p-1}\,dy\bigg)\bigg]^{1+ sp(p-1+\xi) /(d(1+\xi))}\\
&\le c_5\bigg[\rho^{-sp}R^{sp}\left(\frac{l-k}{2}\right)^{-(2-p)}\oint_{I^\ominus_{r+\rho}(t_0)} \oint_{B_{r+\rho}(x_0)}v_{k}(t,x)^{1+\xi}\,dx\,dt\\
&\qquad\quad
+\left[(r+\rho)^{sp}-r^{sp}\right]^{-1}R^{sp}\oint_{I^\ominus_{r+\rho}(t_0)} \oint_{B_{r+\rho}(x_0)}v_{k}(t,x)^{1+\xi}\,dx\,dt\\
&\qquad\quad  +\left(\frac{R}{\rho}\right)^{d+sp}\left(\frac{l-k}{2}\right)^{-1}\left(\oint_{I^\ominus_{r+\rho}(t_0)} \oint_{B_{r+\rho}(x_0)}v_{k}(t,x)^{1+\xi}\,dx\,dt\right)\\
&\qquad\qquad\times \bigg(\sup_{t\in I^\ominus_R(t_0)}\oint_{B_R(x_0)}
u_+(t,y)^{p-1}\,dy\bigg)\bigg]^{1+ sp(p-1+\xi) /(d(1+\xi))}.
\end{align*}
On the other hand, by \eqref{L} again, it holds that
\begin{equation*}
\oint_{I^\ominus_r(t_0)} \oint_{B_r(x_0)} v_{\frac{l+k}{2}}(t,x)^{(p-1+\xi)(1+sp/d)}\,dx\,dt
\ge\left(\frac{l-k}{2}\right)^{(p-1+\xi)(1+sp/d)-(1+\xi)}\oint_{I^\ominus_r(t_0)} \oint_{B_r(x_0)}v_l(t,x)^{1+\xi}\,dx\,dt.
\end{equation*}
Thus,
\begin{align*}
&\left(\frac{l-k}{2}\right)^{(p-1+\xi)(1+sp/d)-(1+\xi)}\oint_{I^\ominus_r(t_0)} \oint_{B_r(x_0)}v_l(t,x)^{1+\xi}\,dx\,dt\\
&\le c_5\bigg[\rho^{-sp}R^{sp}\left(\frac{l-k}{2}\right)^{-(2-p)}
+\left[(r+\rho)^{sp}-r^{sp}\right]^{-1}R^{sp}\\
&\quad\qquad +\left(\frac{R}{\rho}\right)^{d+sp}\left(\frac{l-k}{2}\right)^{-1}\sup_{t\in I^\ominus_R(t_0)}\oint_{B_R(x_0)}
u_+(t,y)^{p-1}\,dy\bigg]^{1+ sp(p-1+\xi) /(d(1+\xi))}\\
&\quad\,\times\left(\oint_{I^\ominus_{r+\rho}(t_0)} \oint_{B_{r+\rho}(x_0)}  v_{k}(t,x)^{1+\xi}\,dx\,dt\right)^{1+sp(p-1+\xi)/(d(1+\xi))}.
\end{align*}

(iii) The plan for the remainder of the proof is to iterate the above estimate. Let us now set up the iteration scheme. Define two sequences $l_i=M(1-2^{-i})$ and $\rho_i=2^{-i-1}R$ for all $i\ge0$, where $M>0$ is to be determined later. Set $r_0=R$, $l_0=0$ and $$r_{i+1}=r_i-\rho_{i+1}=R\left(1+2^{-1-i}\right)/2,\quad i\ge0.$$ In particular, $r_i\downarrow R/2$ and $l_i\uparrow M$ as $i\rightarrow\infty$, and $\left[(r_i+\rho_i)^{sp}-r_i^{sp}\right]^{-1}\le c_6 R^{1-sp}\rho_i^{-1}$ for $i\ge0$.
Define $$A_i=\oint_{I^\ominus_{r_i}(t_0)} \oint_{B_{r_i}(x_0)} v_{l_i}(t,x)^{1+\xi}\,dx\,dt,\quad i\ge0.$$
It holds that for all $i\ge 1$,
\begin{align*}
A_i
&\le c_7\left(\frac{l_i-l_{i-1}}{2}\right)^{-\big((p-1+\xi)(1+sp/d)-(1+\xi)\big)} \\
&\quad\times \bigg[\rho_i^{-sp}R^{sp}\left(\frac{l_i-l_{i-1}}{2}\right)^{-(2-p)}
+ R\rho_i^{-1}\\
&\qquad\qquad
+\left(\frac{R}{\rho_i}\right)^{d+sp}\left(\frac{l_i-l_{i-1}}{2}\right)^{-1}\sup_{t\in I^\ominus_R(t_0)}\oint_{B_R(x_0)}u_+(t,x)^{p-1}\,dx\bigg]^{1+sp(p-1+\xi)/(d(1+\xi))}\\
&\quad \times A_{i-1}^{1+sp(p-1+\xi)/(d(1+\xi))}\\
&\le c_8\left(\frac{2^{i+1}}{M}\right)^{(p-1+\xi)(1+sp/d)-(1+\xi)} \\
&\quad\times \bigg[\frac{2^{(i+1)sp}2^{(i+1)(2-p)}}{M^{2-p}}
+ {2^{i+1}}
+\frac{2^{i+1}2^{(i+1)(d+sp)}}{M}\sup_{t\in I^\ominus_R(t_0)}\oint_{B_R(x_0)}u_+(t,x)^{p-1}\,dx\bigg]^{1+sp(p-1+\xi)/(d(1+\xi))}\\
&\quad\times A_{i-1}^{1+sp(p-1+\xi)/(d(1+\xi))}\\
&\le \frac{c_92^{\gamma i}}{M^{(p-1+\xi)(1+sp/d)-(1+\xi)}} \bigg[\frac{1}{M^{2-p}}+1
+\frac{\sup_{t\in I^\ominus_R(t_0)}\oint_{B_R(x_0)}u_+(t,x)^{p-1}\,dx}{M}\bigg]^{1+sp(p-1+\xi)/(d(1+\xi))}\\
&\quad\times A_{i-1}^{1+sp(p-1+\xi)/(d(1+\xi))},
\end{align*}
where $\gamma=\big((p-1+\xi)(1+sp/d)-(1+\xi)\big) +(d+sp+1)\big(1+sp(p-1+\xi)/(d(1+\xi))\big)$.

Below, we take
$$M\ge \left(\sup_{t\in I^\ominus_R(t_0)}\oint_{B_R(x_0)}u_+(t,x)^{p-1}\,dx\right)^{1/(p-1)}.$$ Then, for all $i\ge1$,
\begin{align*}A_i\le & \frac{c_92^{\gamma i}}{M^{(p-1+\xi)(1+sp/d)-(1+\xi)}} (1+2M^{p-2}) ^{1+sp(p-1+\xi)/(d(1+\xi))} A_{i-1}^{1+sp(p-1+\xi)/(d(1+\xi))}\\
\le&\frac{c_92^{\gamma i}}{M^{(p-1+\xi)(1+sp/d)-(1+\xi)}} \left(1+2\left(\sup_{t\in I^\ominus_R(t_0)}\oint_{B_R(x_0)}u_+(t,x)^{p-1}\,dx\right)^{(p-2)/(p-1)}\right) ^{1+sp(p-1+\xi)/(d(1+\xi))} \\
&\times A_{i-1}^{1+sp(p-1+\xi)/(d(1+\xi))}\\
\le&\frac{c_92^{\gamma i}}{M^{(p-1+\xi)(1+sp/d)-(1+\xi)}} \left(1+2\left(\oint_{I^\ominus_R(t_0)}\oint_{B_R(x_0)}u_+(t,x)^{p-1}\,dx\,d t\right)^{(p-2)/(p-1)}\right) ^{1+sp(p-1+\xi)/(d(1+\xi))} \\
&\times A_{i-1}^{1+sp(p-1+\xi)/(d(1+\xi))},\end{align*} where in the last inequality we used the facts that $p\in (1,2)$ and
$$\oint_{I^\ominus_R(t_0)}\oint_{B_R(x_0)}u_+(t,x)^{p-1}\,dx\,d t\le \sup_{t\in I^\ominus_R(t_0)}\oint_{B_R(x_0)}u_+(t,x)^{p-1}\,dx.$$

For simplicity, in the following we set $\theta>0$ so that
$$(p-1+\xi)(1+sp/d)=(1+\xi)(1+\theta),$$ and write
$$\beta=sp(p-1+\xi)/(d(1+\xi))=sp(1+\theta)/(d(1+sp/d)).$$ Then, the estimate above is reduced
to
$$A_i\le  \frac{c_92^{\gamma i}}{M^{(1+\xi)\theta}}\left(1+2\left(\oint_{I^\ominus_R(t_0)}\oint_{B_R(x_0)}u_+(t,x)^{p-1}\,dx\,d t\right)^{(p-2)/(p-1)}\right) ^{1+\beta} A_{i-1}^ {1+\beta},\quad i\ge1.$$
Now, we take
\begin{align*}
M=&\left(\sup_{t\in I^\ominus_R(t_0)}\oint_{B_R(x_0)}u_+(t,x)^{p-1}\,dx\right)^{1/(p-1)}\\
&\quad+\left\{C_0\left[1+2\left(\oint_{I^\ominus_R(t_0)}\oint_{B_R(x_0)}u_+(t,x)^{p-1}\,dx\,d t\right)^{(p-2)/(p-1)}\right]^{(1+\beta)/\beta}A_0 \right\}^{\beta/((1+\xi)\theta)},
\end{align*}
where $C_0=(2^\gamma)^{1/\beta^2}(c_92^\gamma)^{1/\beta}$.
Then, we have
\begin{align*}
A_0\le &C_0^{-1}M^{(1+\xi)\theta/\beta}\left(1+2\left(\oint_{I^\ominus_R(t_0)}\oint_{B_R(x_0)}u_+(t,x)^{p-1}\,dx\,d t\right)^{(p-2)/(p-1)}\right)^{-(1+\beta)/\beta}\\
=& \left(2^{\gamma}\right)^{-1/\beta^2} \left(\frac{c_92^\gamma}{M^{(1+\xi)\theta}}\left[1+2\left(\oint_{I^\ominus_R(t_0)}\oint_{B_R(x_0)}u_+(t,x)^{p-1}\,dx\,d t\right)^{(p-2)/(p-1)}\right]^{1+\beta}\right)^{-1/\beta}.
\end{align*}
This, along with Lemma \ref{iteration} with $Y_j=A_j$, yields that
$$A_i\rightarrow0,\quad i\rightarrow\infty.$$
Hence, we have
\begin{align*}
&\sup_{B_{R/2}(x_0)\times I^\ominus_{R/2}(t_0)}\left(u-C\int_{I^\ominus_{R/2}(t_0)}\int_{B_R(x_0)^c}\frac{u_+(t,y)^{p-1}}{|x_0-y|^{d+sp}}\,dy\,dt\right)\\
&\le\left(\sup_{t\in I^\ominus_R(t_0)}\oint_{B_R(x_0)}u_+(t,x)^{p-1}\,dx\right)^{1/(p-1)}\\
&\quad+\left\{C_0\left[1+2\left(\oint_{I^\ominus_R(t_0)}\oint_{B_R(x_0)}u_+(t,x)^{p-1}\,dx\,d t\right)^{(p-2)/(p-1)}\right]^{(1+\beta)/\beta}A_0 \right\}^{\beta/((1+\xi)\theta)}.
 \end{align*}
In particular,
noting that $$\left(u(t,x)-C\int^t_{t_0-R^{sp}}\int_{B_R(x_0)^c}\frac{u_+(s,y)^{p-1}}{|x_0-y|^{d+sp}}\,dy\,ds\right)_+^{1+\xi}
\le u_+(t,x)^{1+\xi},$$
we have
\begin{equation}\label{Jj}\begin{split}
&\sup_{B_{R/2}(x_0)\times I^\ominus_{R/2}(t_0)}u\\
&\le\left(\sup_{t\in I^\ominus_R(t_0)}\oint_{B_R(x_0)}u_+(t,x)^{p-1}\,dx\right)^{1/(p-1)} +C{\rm Tail}_{p-1}^{p-1}(u; x_0,R, t_0-R^{sp}, t_0)
\\
&\quad +c_{10}\left[1+2\left(\oint_{I^\ominus_R(t_0)}\oint_{B_R(x_0)}u_+(t,x)^{p-1}\,dx\,d t\right)^{(p-2)/(p-1)}\right]^{(1+\beta)/((1+\xi)\theta)}\\
&\qquad\qquad\times\left(\oint_{I^\ominus_R(t_0)}\oint_{ B_R(x_0)}u_+(t,x)^{1+\xi}\,dx\,dt\right)^{\beta/((1+\xi)\theta)}.
\end{split}\end{equation}

Since $u$ is a subsolution to $\partial_tu-\sL_tu=0$ in $I^\ominus_R(t_0)\times B_R(x_0)$, $u_+$ is also a subsolution to $\partial_tu-\sL_tu=0$ in $I^\ominus_R(t_0)\times B_R(x_0)$ thanks to
Lemma \ref{Z}.
Hence, for the first term in the right hand side of the inequality above, by applying the H\"{o}lder inequality and Proposition \ref{sup} to $u_+$, we get
\begin{align*}
&\left(\sup_{t\in I^\ominus_R(t_0)}\oint_{B_R(x_0)}u_+(t,x)^{p-1}\,dx\right)^{1/(p-1)}\\
&\le\sup_{I^\ominus_R(t_0)}\left(\oint_{B_R(x_0)}u_+(t,x)^p\,dx\right)^{1/p}\\
&\le c_{11}\left(\oint_{I^\ominus_R(t_0)}\oint_{B_R(x_0)}u_+(t,x)^{2p-2}\,dx\,dt+\oint_{I^\ominus_R(t_0)}\oint_{B_R(x_0)}u_+(t,x)^p\,dx\,dt\right)^{1/p}\\
&\quad\,\,\,+
c_{11}{\rm Tail}_{p-1}^{p-1}(u; x_0,R, t_0-R^{sp}, t_0)
\end{align*}
Combining it with \eqref{Jj},
we obtain
\begin{align*}
 \sup_{B_{R/2}(x_0)\times I^\ominus_{R/2}(t_0)}u
&\le c_{12}\left(\oint_{I^\ominus_{R}(t_0)}\oint_{B_{R}(x_0)}u_+(t,x)^{2p-2}\,dx\,dt+\oint_{I^\ominus_{R}(t_0)}\oint_{B_{R}(x_0)}u_+(t,x)^p\,dx\,dt\right)^{1/p}\\
&\quad\,\,\,+c_{12}\left[1+2\left(\oint_{I^\ominus_R(t_0)}\oint_{B_R(x_0)}u_+(t,x)^{p-1}\,dx\,d t\right)^{(p-2)/(p-1)}\right]^{(1+\beta)/((1+\xi)\theta)}\\
&
\qquad\qquad\times\left(\oint_{I^\ominus_R(t_0)}\oint_{ B_R(x_0)}u_+(t,x)^{1+\xi}\,dx\,dt\right)^{\beta/((1+\xi)\theta)}\\
&\quad\,\,\,+c_{12}{\rm Tail}_{p-1}^{p-1}(u; x_0,R, t_0-{R}^{sp}, t_0),
\end{align*}
where $\theta=\frac{(p-1+\xi)(1+sp/d)}{1+\xi}-1$ and $\beta=sp(p-1+\xi)/(d(1+\xi))=sp(1+\theta)/(d(1+sp/d))$.
The proof for the case that $p\in(1,2)$ is complete.
\end{proof}

\bigskip

\noindent {\bf Acknowledgements.}\,\,
 The research of Takashi Kumagai is supported
by JSPS KAKENHI Grant Number JP22H00099 and 23KK0050.
The research of Jian Wang is supported by the National Key R\&D Program of China (2022YFA1006003) and  the National Natural Science Foundation of China (Nos.\ 12071076 and 12225104).

\vskip 0.3truein

{\bf Takashi Kumagai:}
 Department of Mathematics,
Waseda University, Shinjuku, Tokyo, 169-8555, Japan.
\newline \texttt{t-kumagai@waseda.jp}

\bigskip

{\bf Jian Wang:}
    School of Mathematics and Statistics \& Key Laboratory of Analytical Mathematics and Applications (Ministry of Education) \& Fujian Provincial Key Laboratory
of Statistics and Artificial Intelligence, Fujian Normal University, 350117, Fuzhou, P.R. China.
     \newline \texttt{jianwang@fjnu.edu.cn}

\bigskip

{\bf Meng-ge Zhang:}
    School of Mathematics and Statistics, Fujian Normal University,
    350117, Fuzhou, P.R. China.
     \newline \texttt{mgzhangmath@163.com}


\begin{thebibliography}{99}
\bibitem{BK}
Byun, S.S. and Kim, K.: A H\"{o}lder estimate with an optimal tail for nonlocal parabolic $p$-Laplace equations, {\it Annali di Matematica}, {\bf 203} (2024), 109--147.

\bibitem{BLS}
Brasco, L., Lindgren, E. and Str\"omqvist, M.: Continuity of solutions to a nonlinear fractional diffusion equation, {\it  J. Differential Equations}, {\bf 21} (2021), 4319--4381.


\bibitem{CKW1}
Chen, Z.-Q., Kumagai, T. and Wang, J.:
Elliptic Harnack inequalities for symmetric non-local Dirichlet forms,
{\it J. Math. Pures Appl.}, {\bf 125} (2019),  1--42.


\bibitem{CKW2}
Chen, Z.-Q., Kumagai, T. and Wang, J.:
Stability of parabolic Harnack inequalitiess for symmetric non-local Dirichlet forms,
{\it J. Eur. Math. Soc.}, {\bf 22} (2020), 3747--3803.

\bibitem{CM}
Cozzi, M.: Regularity results and Harnack inequalities for minimizers and solutions of nonlocal problems: a unified approach via fractional De Giorgi classes, {\it J. Funct. Anal.}, {\bf 272} (2017), 4762--4837.

\bibitem{DKP1}
Di Castro, A., Kuusi, T. and Palatucci, G.: Local behavior of fractional $p$-minimizers, {\it Ann. Inst. H. Poincar\'e Anal. Non Lin\'{e}aire}, {\bf 33} (2016), 1279--1299.

\bibitem{DKP2}
Di Castro, A., Kuusi, T. and Palatucci, G.: Nonlocal Harnack inequalities, {\it J. Funct. Anal.}, {\bf 267} (2014), 1807--1836.

\bibitem{DPV}
Di Nezza, E., Palatucci, G. and Valdinoci, E.: Hitchhiker's guide to the fractional Sobolev spaces, {\it Bull. Sci. Math.}, {\bf 136} (2012), 521--573.

\bibitem{DZZ}
Ding, M.Y., Zhang, C. and Zhou, S.L.: Local boundedness and H\"{o}lder continuity for the parabolic fractional $p$-Laplace equations, {\it Calc. Var.}, {\bf 60} (2021), Paper no. 38.


\bibitem{GG}
Giaquinta, M. and Giusti, E.: On the regularity of the minima of variational integrals, {\it Acta Math.}, {\bf 148} (1982), 31--46.

\bibitem{Gi}
Giusti, E.: {\it Direct Methods in the Calculus of Variations},  World Scientific, 2003.


\bibitem{KW2}
Kassmann, M. and Weidner, M.: The parabolic Harnack inequality for nonlocal equations, to appear in {\it Duke J. Math.}, see arXiv: 2303.05975

\bibitem{MRT}
Maz\'{o}n, J.M., Rossi, J.D. and Toledo, J.: Fractional p-Laplacian evolution equations, {\it J. Math. Pures Appl.},
 {\bf 105} (2016), 810--844.

\bibitem{PZ}
Prasad, H. and Tewary, V.: Local boundedness of variational solutions to nonlocal double phase parabolic equations, {\it J. Differential Equations}, {\bf 351} (2023), 243--276.

\bibitem{St1}
Str\"{o}mqvist, M.: Local boundedness of solutions to non-local parabolic equations modeled on the fractional $p$-Laplacian, {\it J. Differential Equations}, {\bf 266} (2019), 7948--7979.

\bibitem{St2}
Str\"{o}mqvist, M.: Harnack's inequality for parabolic nonlocal equations, {\it Ann. Inst. H. Poincar\'e Anal. Non Lin\'{e}aire}, {\bf 36} (2019), 1709--1745.

\bibitem{VJL}
V\'{a}zquez, J.L.: The Dirichlet problem for the fractional $p$-Laplacian evolution equation, {\it  J. Differential Equations}, {\bf 260} (2016), 6038--6056.

\end{thebibliography}
\end{document}